\documentclass[11pt]{amsart} 

\usepackage{fullpage}
\usepackage{url,amssymb,amsmath,amsfonts,amsthm,mathrsfs} %,diagbox}

\usepackage{color}

\definecolor{webcolor}{rgb}{0.8,0,0.2}
\definecolor{webbrown}{rgb}{.6,0,0}
\usepackage[
        colorlinks,
        linkcolor=webbrown,  filecolor=webcolor,  citecolor=webbrown, 
        backref,
        pdfauthor={David Zywina}, 
       pdftitle={},
]{hyperref}
\usepackage[alphabetic,backrefs,lite]{amsrefs} % for bibliography

\numberwithin{equation}{section}

\renewcommand{\AA}{\mathbb A}

\newcommand{\CC}{\mathbb C}

\newcommand{\FF}{\mathbb F}

\newcommand{\PP}{\mathbb P}
\newcommand{\QQ}{\mathbb Q}

\newcommand{\ZZ}{\mathbb Z} 
\newcommand{\Zhat}{\widehat\ZZ}

\newcommand{\calG}{\mathcal G}  \newcommand{\calF}{\mathcal F}
\newcommand{\calH}{\mathcal H}

\newcommand{\calJ}{\mathcal J}

\newcommand{\calS}{\mathcal S}

\newcommand{\calA}{\mathcal A}

\newcommand{\calX}{\mathcal X}

\newcommand{\p}{\mathfrak p}

\newcommand{\scrE}{\mathscr E}

\def\cyc{{\operatorname{cyc}}}
\def\ab{{\operatorname{ab}}}

\def\Spec{\operatorname{Spec}}

\def\Gal{\operatorname{Gal}}
 
\def \GL {\operatorname{GL}}

\def \SL {\operatorname{SL}}
\def \PSL {\operatorname{PSL}}

\def\Aut{\operatorname{Aut}}

\newcommand{\defi}[1]{\textsf{#1}} % for defined terms

\newcommand\blank[1]{}

\def\bbar#1{\setbox0=\hbox{$#1$}\dimen0=.2\ht0 \kern\dimen0 
\overline{\kern-\dimen0 #1}}
\newcommand{\Qbar}{{\overline{\mathbb Q}}} 
\newcommand{\Kbar}{{\bbar{K}}}

\newtheorem{thm}{Theorem}[section]
\newtheorem{lemma}[thm]{Lemma}

\newtheorem{prop}[thm]{Proposition}
\newtheorem{conj}[thm]{Conjecture}

\theoremstyle{definition}

\theoremstyle{remark}
\newtheorem{remark}[thm]{Remark}
\newtheorem{example}[thm]{Example}

\newenvironment{romanenum}{\hfill \begin{enumerate} }{\end{enumerate}}
\newenvironment{alphenum}{\hfill \begin{enumerate} }{\end{enumerate}}

\begin{document}

\title{Open image computations for elliptic curves over number fields}
%% MSC-class: 11G05 (Primary) 11F80 (Secondary)
%%\keywords{}
\author{David Zywina}
\address{Department of Mathematics, Cornell University, Ithaca, NY 14853, USA}
\email{zywina@math.cornell.edu}
%\urladdr{http://www.math.cornell.edu/~zywina}

\begin{abstract}
For a non-CM elliptic curve $E$ defined over a number field $K$, the Galois action on its torsion points gives rise to a Galois representation $\rho_E\colon \Gal(\Kbar/K)\to \GL_2(\Zhat)$ that is unique up to isomorphism.   A renowned theorem of Serre says that the image of $\rho_E$ is an open, and hence finite index, subgroup of $\GL_2(\Zhat)$.  In an earlier work of the author, an algorithm was given that computed the image of $\rho_E$  up to conjugacy in $\GL_2(\Zhat)$ in the special case $K=\QQ$.   A fundamental ingredient of this earlier work was the Kronecker--Weber theorem whose conclusion fails for number fields $K\neq \QQ$.   We shall give an overview of an analogous algorithm for a general number field and work out the required group theory.  We also give some bounds on the index in Serre's theorem for a typical elliptic curve over a fixed number field.
\end{abstract}

\maketitle

\section{Introduction}
\subsection{Serre's open image theorem}

Let $E$ be an elliptic curve defined over a number field $K$.  We denote its $j$-invariant by $j_E$.  For each integer $N>1$, let $E[N]$ be the $N$-torsion subgroup of $E(\Kbar)$, where $\Kbar$ is a fixed algebraic closure of $K$.  The group $E[N]$ is a free $\ZZ/N\ZZ$-module of rank $2$.    The absolute Galois group $\Gal_K:=\Gal(\Kbar/K)$ acts on $E[N]$ and respects the group structure.   We may express this Galois action in terms of a representation $\rho_{E,N}\colon \Gal_K\to \Aut(E[N])\cong \GL_2(\ZZ/N\ZZ)$.  By choosing compatible bases and taking the inverse limit, these representations combine into a single Galois representation 
\[
\rho_E\colon \Gal_K \to \GL_2(\Zhat),
\]
where $\Zhat$ is the profinite completion of $\ZZ$.  The representation $\rho_E$ is uniquely determined up to isomorphism and hence the image $\rho_E(\Gal_K)$ is uniquely determined up to conjugacy in $\GL_2(\Zhat)$.    With respect to the profinite topologies, we find that $\rho_E$ is continuous and hence $\rho_E(\Gal_K)$ is a closed subgroup of the compact group $\GL_2(\Zhat)$. 

In \cite{Serre-Inv72}, Serre proved the following theorem which says that, up to finite index, the image of $\rho_E$ is as large as possible when $E$ is non-CM.

\begin{thm}[Serre's open image theorem] \label{T:Serre 1972}
Let $E$ be a non-CM elliptic curve defined over a number field $K$.  Then $\rho_E(\Gal_K)$ is an open subgroup of $\GL_2(\Zhat)$.  Equivalently, $\rho_E(\Gal_K)$ is a finite index subgroup of $\GL_2(\Zhat)$.
\end{thm}

Let $E$ be a non-CM elliptic curve over a number field $K$.    We will find it convenient to instead work with the dual representation
\[
\rho_E^*\colon \Gal_K\to \GL_2(\Zhat)
\]
of $\rho_E$, i.e., $\rho_E^*(\sigma)$ is the transpose of $\rho_E(\sigma^{-1})$.   Similarly, we can define $\rho_{E,N}^*$ for each $N\geq 1$.  

Define the group $G_E:=\rho_E^*(\Gal_K)$.   The group $G_E$ is open in $\GL_2(\Zhat)$ by Theorem~\ref{T:Serre 1972} and is uniquely determined up to conjugacy in $\GL_2(\Zhat)$.  Of course, computing $G_E$ is equivalent to computing the image of $\rho_E$ since $\rho_E(\Gal_K)=\{A^t: A \in G_E\}$.
The group $G_E$, when known, will have a simple description since it is open in $\GL_2(\Zhat)$, i.e., it is given by its level $N$ and a set of generators for its image modulo $N$ in $\GL_2(\ZZ/N\ZZ)$.  

Unfortunately, Serre's proof is ineffective in general.   In the special case $K=\QQ$, the author has recently given, and fully implemented, an algorithm to compute $G_E$ up to conjugacy, see \cite{openimage}.   The images for all non-CM elliptic curves over $\QQ$ of conductor at most $500000$ are easily accessible via the LMFDB \cite{lmfdb}.     

The goal of this article is to begin the study of how to compute the groups $G_E$ for a general number $K$.  A vital ingredient in the arguments of \cite{openimage} is that the commutator subgroup of $G_E$ agrees with $G_E\cap \SL_2(\Zhat)$ when $K=\QQ$; this makes use of the Kronecker--Weber theorem, cf.~\S\ref{SS:cyclotomic constraint}.   When $K\neq \QQ$, the commutator subgroup of $G_E$ is often strictly smaller than $G_E\cap \SL_2(\Zhat)$.  Much of this paper is dedicated to dealing with the new group theoretic complications that arise for this reason.

Fix a number field $K$.   We shall give an algorithm which defines a finite set $J_K\subseteq K$ and computes the group $G_E$, up to conjugacy, for all non-CM elliptic curves $E$ over $K$ whose $j$-invariant does not lie in $J_K$ and for which $\rho_{E,\ell}(\Gal_K)\supseteq \SL_2(\ZZ/\ell\ZZ)$ holds for primes $\ell>19$ (conjecturally the condition on the images of the $\rho_{E,\ell}$ can be removed by making $J_K$ large enough, cf.~Conjecture~\ref{C:uniformity}).  

Our set $J_K$ and our algorithm will both depend only on a \emph{finite} number of modular curves and morphisms which do not depend on $K$; these curves and morphisms can thus be precomputed.   Given any value in $K$, we will be able to determine whether or not it lies in $J_K$ (explicitly giving the full set is much harder since its finiteness makes use of Faltings' theorem).   The group theoretic aspects of the algorithm have been fully implemented.  The modular curve computations now seem reasonable but have not been performed yet.  Indeed, one of the main goals of this work was to confirm there were not too many cases as to make the modular curve computations infeasible.

\subsection{Index bounds}

For a non-CM $E/K$, the group $\det(G_E)$ depends only on $K$, cf.~\S\ref{SS:cyclotomic constraint}.  Since $[\GL_2(\Zhat):G_E]=[\Zhat^\times:\det(G_E)] \cdot [\SL_2(\Zhat):G_E \cap\SL_2(\Zhat)]$, we will thus focus on the index $[\SL_2(\Zhat):G_E \cap\SL_2(\Zhat)]$ when $K$ is fixed.

\begin{thm} \label{T:indices}
Let $K$ be a number field. There is a finite set $J_K \subseteq K$ such that for any non-CM elliptic curve $E$ over $K$ with  $j_E \notin J_K$ and $\rho_{E,\ell}(\Gal_K)\supseteq \SL_2(\ZZ/\ell\ZZ)$ for all primes $\ell>19$, we have
\[
[\SL_2(\Zhat): G_E \cap \SL_2(\Zhat)] \leq \begin{cases}
      1382400, &  \\
      677376 & \text{if $K\not\supseteq \QQ(\sqrt{-1},\sqrt{2},\sqrt{3},\sqrt{5})$},\\
      172800& \text{if $K\cap \QQ(\sqrt{-1})=\QQ$},\\
      30000& \text{if $K \cap \QQ(\sqrt{-1},\sqrt{2},\sqrt{3})=\QQ$},\\      
      7200& \text{if $K \cap \QQ(\sqrt{-1},\sqrt{2},\sqrt{3},\sqrt{5},\sqrt{7},\sqrt{11})=\QQ$},\\            
	1536 & \text{if $K=\QQ$.}      
\end{cases}
\]
\end{thm}

We will prove Theorem~\ref{T:indices} in \S\ref{SS:agreeable closure overview} where we reduce it to a direct calculation involving a finite number of open subgroups of $\GL_2(\Zhat)$ coming from our group theoretic computations.

\begin{remark}
\begin{romanenum}
\item
A well-known uniformity conjecture (Conjecture~\ref{C:uniformity}) would imply that Theorem~\ref{T:indices} still holds after removing the assumption on the images of the $\rho_{E,\ell}$.
\item
The index $[\SL_2(\Zhat): G_E \cap \SL_2(\Zhat)]$ can become arbitrarily large as we vary over all number fields $K$ and all non-CM elliptic curves $E/K$.  For example, consider a fixed non-CM elliptic curve $E/\QQ$ base extended by $K:=\QQ(E[2^i])$ with integers $i\geq 0$.    Theorem~\ref{T:indices}, along with Conjecture~\ref{C:uniformity}, shows that large indices are rare when the number field $K$ is fixed.
\end{romanenum}
\end{remark}

\subsection{Notation}

We now give some notation that will hold throughout.   All profinite groups will be viewed as topological groups with their profinite topology.  In particular, finite groups will have the discrete topology.   For a topological group $G$, we define its \defi{commutator subgroup} $[G,G]$ to be the smallest closed normal subgroup of $G$ for which $G/[G,G]$ is abelian.  Equivalently, $[G,G]$ is the closed subgroup of $G$ generated by the set of commutators $\{ghg^{-1}h^{-1}: g,h \in G\}$.

For each integer $N>1$, we let $\ZZ_N$ be the ring obtained by taking the inverse limit of the $\ZZ/N^e\ZZ$ with $e\geq 1$. Let $\Zhat$ be the ring obtained taking the inverse limit of $\ZZ/n\ZZ$ over all positive integers $n$ ordered by divisibility.  With the profinite topology, $\ZZ_N$ and $\Zhat$ are compact topological rings.  We have natural isomorphisms 
\[
\ZZ_N = {\prod}_{\ell|N} \ZZ_\ell\quad  \text{and} \quad \Zhat= \ZZ_N \times {\prod}_{\ell \nmid N} \ZZ_\ell ={\prod}_{\ell}\ZZ_\ell,
\]   
where the products are over primes $\ell$.   The symbol $\ell$ will always denote a rational prime.     Fix a positive integer $n$ dividing a power of $N$.  For a subgroup $G$ of $\GL_2(\ZZ_N)$ or $\GL_2(\Zhat)$, let $G_n\subseteq \GL_2(\ZZ_n)$ be the group that is the image of $G$ under the $n$-adic projection map.  From context it should be clear when we have $G_i$ with an index $i$ instead.

The \defi{level} of an open subgroup $G$ of $\GL_2(\Zhat)$ is the smallest positive integer $n$ for which $G$ contains the kernel of the reduction modulo $n$ homomorphism $\GL_2(\Zhat)\to \GL_2(\ZZ/n\ZZ)$.  The \defi{level} of an open subgroup $G$ of $\GL_2(\ZZ_N)$ is the smallest positive integer $n$ that divides some power of $N$ and for which $G$ contains the kernel of the reduction modulo $n$ homomorphism $\GL_2(\ZZ_N)\to \GL_2(\ZZ/n\ZZ)$. Similarly, we can define the \defi{level} of open subgroups of $\SL_2(\Zhat)$ and $\SL_2(\ZZ_N)$.

We shall view $\Qbar$ and any other algebraic field extension  of $\QQ$ that arise as subfields of $\CC$ (this is mainly to ensure easy comparison with the analytic theory when working with modular curves).

\subsection{Ideas underlying the algorithm}

We now briefly give some ideas and motivation behind our algorithm to compute the groups $G_E$.  Full details will be presented in \S\ref{S:overview} and this section will not be used later on.

Fix a number field $K$ and let $K^\cyc \subseteq \Kbar$ be the cyclotomic extension of $K$.  Let $\chi_\cyc\colon \Gal_\QQ \to  \Zhat^\times$ be the cyclotomic character, cf.~\S\ref{SS:cyclotomic constraint}. 

\subsubsection{An alternate description of the image}
Consider a non-CM elliptic curve $E$ over $K$.  We shall give an expression for $G_E$ in terms of a triple $(\calG,G,\gamma_E)$, see (\ref{E:GE presentation}) below.   This description of $G_E$ will reduce its computation to the computation of the triple.  The first step is to consider a larger group 
\[
G_E \subseteq \calG \subseteq \GL_2(\Zhat)
\]
 called the \defi{agreeable closure} of $G_E$, cf.~\S\ref{SS:agreeable basics}.   This group $\calG$ contains all the scalar matrices of $\GL_2(\Zhat)$ and satisfies $[\calG,\calG]=[G_E,G_E]$.   The group $\calG$ also has the advantage that its level is often significantly smaller than the level of $G_E$.  We have inclusions
 \[
 [\calG,\calG] = [G_E,G_E] \subseteq G_E \cap \SL_2(\Zhat) \subseteq \calG \cap \SL_2(\Zhat)
 \]
of open subgroups of  $\SL_2(\Zhat)$.   When $K=\QQ$, we will always have $G_E \cap \SL_2(\Zhat)=[G_E,G_E]=[\calG,\calG]$, cf.~Lemma~\ref{L:KW}; this need not hold when $K\neq \QQ$.
 
 We now fix an open subgroup $G$ of $\calG$ for which $G \cap \SL_2(\Zhat) = G_E \cap \SL_2(\Zhat)$.   The group $G$ is normal in $\calG$ with $\calG/G$ a finite abelian group since $G$ contains $[\calG,\calG]$ and is open in $\calG$.  Let
 \begin{align*}
\alpha_{G,E}\colon \Gal_K \xrightarrow{\rho^*_E} G_E \subseteq \calG \to \calG/G
 \end{align*}
 be the homomorphism obtained by composing $\rho_E^*$ with the obvious quotient map.   
 
 By Lemma~\ref{L:KW}, we have $G \cap \SL_2(\Zhat)=G_E\cap \SL_2(\Zhat)=\rho_E^*(\Gal(\Kbar/K^\cyc))$.  Therefore, $\alpha_{G,E}$ factors as $\Gal_K\to \Gal(K^\cyc/K) \to \calG/G$, where the first map is the quotient map.  There is thus a unique homomorphism 
\[
\gamma_E \colon \chi_\cyc(\Gal_K) \to \calG/G
\]
that satisfies $\alpha_{G,E}(\sigma)=\gamma_E(\chi_\cyc(\sigma))^{-1}$ for all $\sigma \in \Gal_K$.

We claim that 
\begin{align} \label{E:GE presentation}
G_E=\big\{ g \in \calG : \det g \in \chi_{\cyc}(\Gal_K), \, gG = \gamma_E(\det g) \big\}.
\end{align}
Denote by $\calH_E$ the group on the right-hand side of (\ref{E:GE presentation}).  For any $\sigma \in \Gal_K$, the above definitions give $\rho_E^*(\sigma) G =\gamma_E(\chi_\cyc(\sigma))^{-1}$.  As observed in \S\ref{SS:cyclotomic constraint}, we have $\det \circ \rho_E^* = \chi_\cyc^{-1}$ and hence $\rho_E^*(\sigma) G =\gamma_E(\det(\rho_E^*(\sigma)))$ for all $\sigma\in \Gal_K$.  This proves the inclusion $G_E \subseteq \calH_E$.   Since $G_E \subseteq \calH_E$ and $\calH_E\cap \SL_2(\Zhat)=G\cap \SL_2(\Zhat)=G_E\cap \SL_2(\Zhat)$, to prove the claim it suffices to show that $\det(G_E) \supseteq \det(\calH_E)$.  The claim thus follows since $\det(G_E)=\chi_\cyc(\Gal_K) \supseteq \det(\calH_E)$.

\subsubsection{A computation perspective}

Let $E$ be a non-CM elliptic curve over $K$. For simplicity, let us assume that $\rho_{E,\ell}(\Gal_K) \supseteq \SL_2(\ZZ/\ell\ZZ)$ for all primes $\ell>19$.      We wish to use (\ref{E:GE presentation}) to compute $G_E$.  We first need to find the agreeable closure $\calG$ of $G_E$.

We will see in \S\ref{SS:agreeable closure overview} that only finitely many $\calG$ can occur as we vary over all such elliptic curves $E/K$ with our fixed number field $K$.   Moreover, we shall prove that there is a finite set $J_K \subseteq K$, depending only on $K$, and a finite set of open subgroups $\calA_1$ of $\GL_2(\Zhat)$ such that if $j_E\notin J_K$, then $\calG$ is conjugate in $\GL_2(\Zhat)$ to a unique group in $\calA_1$.   Our set $\calA_1$ does not depend on $K$ or the curve $E/K$ and has cardinality $11960$.  One of the main tasks of this paper is to explicitly compute such a set $\calA_1$.   

We shall now impose the additional condition that $j_E\notin J_K$ and hence $\calG$ is conjugate in $\GL_2(\Zhat)$ to a unique group in $\calA_1$.  The group $\calG$ will in fact be conjugate to the group $\calG' \in \calA_1$ with maximal index $[\GL_2(\Zhat):\calG']$ so that $G_E$ is conjugate in $\GL_2(\Zhat)$ to a subgroup of $\calG'$.   Whether $G_E$ is conjugate to a subgroup of $\calG'$ can be readily determined after computing the modular curve $X_{\calG'}$ and its morphism to the $j$-line, cf.~\S\ref{SS:modular curves}.   Since $\calA_1$ is finite, the curve $X_{\calG'}$ and its morphism to the $j$-line can be precomputed for all $\calG' \in \calA_1$.  This gives the general idea of how to compute $\calG$; it is only unique determined up to conjugacy in $\GL_2(\Zhat)$ which is fine since the group $G_E$ also has this property.

Now assume that we know the agreeable closure $\calG$ of $G_E$.   For each subgroup $[\calG,\calG] \subseteq H \subseteq \calG \cap \SL_2(\Zhat)$, choose an open subgroup $G$ of $\calG$ for which $G\cap \SL_2(\Zhat)=H$ and define the character
\[
\alpha_{G,E}\colon \Gal_K \xrightarrow{\rho_E^*} \calG \to \calG/G.
\]
We will use modular curves and specializations to compute $\alpha_{G,E}$; these modular curves can also be precomputed.  

Amongst all the $G$, we find one with maximal index $[\SL_2(\Zhat): G\cap \SL_2(\Zhat)]$ for which $\alpha_{G,E}(\Gal(\Kbar/K^\cyc))=1$;  this group will satisfy $G\cap \SL_2(\Zhat)=G_E \cap \SL_2(\Zhat)$.  Now that we have $\calG$, $G$ and $\alpha_{G,E}$, we let $\gamma_E\colon \chi_\cyc(\Gal_K)\to \calG/G$ be the character for which $\alpha_{G,E}(\sigma)=\gamma_E(\chi_{\cyc}(\sigma))^{-1}$ for all $\sigma\in \Gal_K$.    So after conjugating appropriately in $\GL_2(\Zhat)$, we deduce that (\ref{E:GE presentation}) holds with our computable $\calG$, $G$ and $\alpha_{G,E}$.

\subsection{Acknowledgements}

The computations in this paper were performed using the Magma computer algebra system \cite{Magma}.  Many thanks to the referees for their useful comments and suggestions.

\section{Overview of ideas and group theoretic results} \label{S:overview}

\subsection{Cyclotomic and abelian extensions} \label{SS:cyclotomic constraint}

Let  $\chi_\cyc\colon \Gal_\QQ \to \Zhat^\times$ be the \defi{cyclotomic character}, i.e., the continuous homomorphism such that for every integer $n\geq 1$ and every $n$-th root of unity $\zeta\in \Qbar$ we have $\sigma(\zeta)=\zeta^{\chi_\cyc(\sigma)\bmod{n}}$ for all $\sigma\in \Gal_\QQ$.

Fix a number field $K$.  We can identify $K$ with a subfield of $\Qbar$ and take $\Kbar=\Qbar$.   Let $K^\cyc$ be the cyclotomic extension of $K$ in $\Kbar$.  Let $K^\ab$ be the maximal abelian extension of $K$ in $\Kbar$.   We have an inclusion $K^\cyc \subseteq K^\ab$.  We have $K^\cyc \subsetneq K^\ab$ whenever $K\neq \QQ$.  The \emph{Kronecker--Weber theorem} says that $\QQ^\cyc = \QQ^\ab$. 

 Let $E$ be a non-CM elliptic curve over $K$.    Using the Weil pairing on $E[n]$ for $n\geq 1$, one finds that $\det\circ \rho_E^*=\chi_\cyc^{-1}|_{\Gal_K}$.  In particular, $\det(G_E)=\chi_\cyc(\Gal_K)$ is an open subgroup of $\Zhat^\times$ that depends only on $K$.

 \begin{lemma} \label{L:KW}
 \begin{romanenum}
 \item  \label{L:KW i}
We have $\rho_E^*(\Gal_{K^\cyc})=G_E\cap \SL_2(\Zhat)$ and  $\rho_E^*(\Gal_{K^\ab})=[G_E,G_E]$. 
\item \label{L:KW ii}
We have an inclusion $[G_E,G_E]\subseteq G_E \cap \SL_2(\Zhat)$.
\item \label{L:KW iii}
If $K=\QQ$, then $[G_E,G_E]=G_E \cap \SL_2(\Zhat) $. 
\end{romanenum}
\end{lemma}
\begin{proof}
We  have $\rho_E^*(\Gal(\Kbar/K^\ab))= [G_E,G_E]$ since $\Gal(\Kbar/K^\ab)$ is the commutator subgroup of $\Gal_K$.    Since $\chi_\cyc^{-1}=\det\circ \rho_E^*$, we have $\rho_E^*(\Gal(\Kbar/K^\cyc))=G_E\cap \SL_2(\Zhat)$.   The inclusion $[G_E,G_E]\subseteq G_E \cap \SL_2(\Zhat)$ thus follows from $K^\cyc \subseteq K^\ab$.  We have equality when $K=\QQ$ since $\QQ^\cyc=\QQ^\ab$.
\end{proof}

\begin{example} \label{ex:not surjective}
Consider any number field $K \neq \QQ$.   The main result of \cite{MR2721742} implies that for a ``random'' elliptic curve $E/K$, we have $G_E\supseteq \SL_2(\Zhat)$.  For such an elliptic curve $E/K$, we have
\[
[G_E,G_E] \subsetneq \SL_2(\Zhat)=G_E \cap \SL_2(\Zhat),
\]
where $[G_E,G_E] \neq \SL_2(\Zhat)$ can be shown by noting that the commutator subgroup of $\GL_2(\ZZ/2\ZZ)$ is a proper subgroup of $\SL_2(\ZZ/2\ZZ)$.

Now suppose that $K=\QQ$.   Since $[G_E,G_E]=G_E \cap \SL_2(\Zhat)$ by Lemma~\ref{L:KW}(\ref{L:KW iii}) and $[G_E,G_E]\subsetneq \SL_2(\Zhat)$, we find that $G_E\not\supseteq \SL_2(\Zhat)$.  That $\rho_E\colon \Gal_\QQ\to \GL_2(\Zhat)$ is never surjective was first observed by Serre, cf.~Proposition~22 of \cite{Serre-Inv72}. 
\end{example}

\subsection{Modular curves} \label{SS:modular curves}

Let $G$ be an open subgroup of $\GL_2(\Zhat)$ that contains $-I$.   We define $K_G$ to be the unique subfield of $\QQ^\cyc$ for which $\chi_\cyc(\Gal(\Qbar/K_G))=\det(G)$; it is a number field due to the openness of $G$.  

 Associated to $G$, there is a \defi{modular curve} $X_G$; it is a smooth projective and geometrically irreducible curve defined over $K_G$ that comes with a morphism 
\[
\pi_G\colon X_G \to \PP^1_{K_G}=\AA_{K_G}^1 \cup \{\infty\}.
\]
We define the \defi{genus} of the group $G$ to be the genus of the curve $X_G$.  For our applications, the following property of these curves are key.

\begin{prop} \label{P:key property}
For any number field $K\subseteq \Qbar$ and non-CM elliptic curve $E/K$, $\rho_E^*(\Gal_K)$ is conjugate in $\GL_2(\Zhat)$ to a subgroup of $G$ if and only if $K\supseteq K_G$ and $j_E \in \pi_G(X_G(K))$.
\end{prop}

\begin{remark} \label{R:modular curve follow-up}
Modular curves are fully discussed in \S3 of \cite{openimage} with the additional assumption $\det(G)=\Zhat^\times$ (equivalently, $K_G=\QQ$).   We now make some remarks indicating that everything in \S3 of \cite{openimage} carries over straightforwardly to the general setting.  Let $N$ be a positive integer divisible by the level of $G$ and let $\bbar{G}\subseteq \GL_2(\ZZ/N\ZZ)$ be the image of $G$ modulo $N$.   With notation as in \cite[\S3]{openimage}, we define $X_G$ to be the smooth, projective and geometrically irreducible curve over $K_G$ whose function field is $\calF_N^{\bbar{G}}$; this field indeed has transcendence degree $1$ over $\QQ$ and the number field $K_G$ is the algebraical closure of $\QQ$ in $\calF_N^{\bbar{G}}$.   The field $K_G(X_G)=\calF_N^{\bbar{G}}$ consists of modular functions of level $N$ and contains the modular $j$-invariant $j$.  The inclusion of function fields $K_G(X_G)\supseteq K_G(j)$ induces a dominant morphism $\pi_G\colon X_G \to  \PP^1_{K_G}=\Spec K_G[j] \cup \{\infty\}=\AA^1_{K_G} \cup \{\infty\}$ of curves over $K_G$.   Since $\det(G_E)=\chi_\cyc(\Gal_K)$, we have $\det(G_E)\subseteq \det(G)$ if and only if $K\supseteq K_G$.  Proposition~\ref{P:key property} is proved in the same manner as \cite[Proposition 6.4]{openimage} by working over $K_G$ instead of $\QQ$.

Let $\Gamma_G$ be the congruence subgroup consisting of matrices in $\SL_2(\ZZ)$ whose image modulo $N$ lies in $\bbar{G}$.  With our implicit embedding $K_G\subseteq \CC$, there is an isomorphism of smooth compact Riemann surfaces between $X_G(\CC)$ and $\calX_{\Gamma_G}:=\Gamma_G\backslash\calH^*$, where $\calH^*$ is the extended complex upper half-plane and $\Gamma_G$ acts on it by linear fractional transformations (as noted in \S3.3 of \cite{openimage}, they have the same function field).  In particular, the genus of $\Gamma_G$, i.e., the genus of $\calX_{\Gamma_G}$, agrees with the genus of $G$.
\end{remark}

There are many approaches to computing models of $X_G$.   In \S5 of \cite{openimage}, we give an algorithm for computing a model of $X_G$ using modular forms under the assumption $\det(G)=\Zhat^\times$.  This assumption is not needed with the only change being that the spaces of modular forms $M_{k,G}$ that arise in \cite{openimage} are now vector spaces over $K_G$ instead of $\QQ$.    For our exposition below, we will take it for granted that one can always compute a model of $X_G$ in terms of modular forms/functions.   With respect to such a model, one can also compute the morphism $\pi_G$.

\subsection{The agreeable closure} \label{SS:agreeable closure overview}

Our initial strategy for computing the group $G_E$ is to instead compute a larger group that has the same commutator subgroup.  We first define the class of groups we will consider.  

We say that a subgroup $\calG$ of $\GL_2(\Zhat)$ is \defi{agreeable} if it is open, contains all the scalar matrices in $\GL_2(\Zhat)$, and each prime dividing the level of $\calG$ also divides the level of $[\calG,\calG]\subseteq \SL_2(\Zhat)$.   This definition uses that $[\calG,\calG]$ is open in $\SL_2(\Zhat)$, cf.~Lemma~\ref{L:openness of commutator}.  For any open subgroup $G$ of $\GL_2(\Zhat)$, there is a unique minimal agreeable subgroup $\calG$ of $\GL_2(\Zhat)$ that satisfies $G\subseteq \calG$, cf.~\S\ref{SS:agreeable basics}.  We call $\calG$ the \defi{agreeable closure} of $G$.   We have $[G,G]=[\calG,\calG]$, cf.~Lemma~\ref{L:agreeable commutators agree}.

Let $\calA_1'$ be the set of subgroups of $\GL_2(\Zhat)$ that are agreeable and have genus at most $1$.  The set $\calA_1'$ is stable under conjugation by $\GL_2(\Zhat)$.   Let $\calA_1$ be a set of representatives of the $\GL_2(\Zhat)$-conjugacy classes of $\calA_1'$.

\begin{thm} \label{T:main agreeable 1}
The set $\calA_1$ is finite and computable.  For all $G\in \calA_1$, the level of $[G,G]\subseteq \SL_2(\Zhat)$ is not divisible by any prime $\ell>19$.
\end{thm}

Theorem~\ref{T:main agreeable 1}, along with the other theorems in \S\ref{SS:agreeable closure overview}, will be proved in \S\ref{S:proofs of agreeable closure overview}.  The groups in $\calA_1$, up to conjugacy, can be found in the public repository \cite{newgithub}.  The set $\calA_1$ has cardinality $11960$.    The number of groups $G\in \calA_1$ in terms of their genus and the index $[\Zhat^\times:\det(G)]$ is given in Table~\ref{table-A1}.   The largest integer that occurs as the level of a group in $\calA_1$ is $1176$.   

\begin{table}[h]
\begin{tabular}{l |lllll}
   & $1$ & $2$ & $2^2$ & $2^3$ & $2^4$ \\
\hline  genus $0$ &   $418$  &   $1490$  &  $1319$     &  $417$     &   $38$    \\
genus  $1$ &  $1078$   &  $3379$   &   $2891$    &   $866$    &      $64$
\end{tabular}
\caption{Number of groups $G$ in $\calA_1$ broken up by the genus and the index of $\det(G)$ in $\Zhat^\times$.}
\label{table-A1}
\end{table}

%{* [r`genus,r`det_index] : r in Families | r`is_agreeable *};
%{*
%    [ 0, 1 ]^^418,
%    [ 0, 2 ]^^1490,
%    [ 0, 4 ]^^1319,
%    [ 0, 8 ]^^417,
%    [ 0, 16 ]^^38,
%    [ 1, 1 ]^^1078,
%    [ 1, 2 ]^^3383,
%    [ 1, 4 ]^^2897,
%    [ 1, 8 ]^^868,
%    [ 1, 16 ]^^64
%*}

%> {*r`N: r in Families | r`is_agreeable *};                                
%{* 1, 2^^3, 3^^10, 4^^29, 5^^17, 6^^54, 7^^19, 8^^338, 9^^36, 10^^40, 11^^9, 
%12^^523, 13^^5, 14^^20, 15^^69, 16^^343, 17^^4, 18^^37, 19^^4, 20^^187, 21^^38, 
%22, 24^^3382, 25^^3, 26^^5, 27^^8, 28^^84, 30^^67, 32^^49, 33^^11, 36^^158, 
%39^^7, 40^^877, 42^^26, 44^^15, 48^^396, 49^^2, 51^^4, 52^^22, 56^^377, 57^^4, 
%60^^412, 66, 68^^7, 72^^669, 75^^3, 76^^6, 78^^5, 84^^139, 88^^55, 96^^49, 
%100^^5, 104^^102, 108^^12, 120^^2141, 132^^18, 136^^27, 147^^2, 152^^22, 
%156^^25, 168^^632, 196^^3, 200^^19, 204^^7, 216^^44, 228^^6, 264^^66, 300^^5, 
%312^^113, 392^^11, 408^^27, 456^^22, 588^^3, 600^^19, 1176^^11 *}
%> {*r`degree: r in Families | r`is_agreeable *};                           
%{* 1^^10, 2^^22, 3^^26, 4^^26, 5^^32, 6^^267, 7^^10, 8^^111, 9^^32, 10^^134, 
%11^^10, 12^^703, 14^^98, 15^^147, 16^^150, 18^^644, 20^^437, 21^^58, 22^^22, 
%24^^1305, 27^^48, 28^^246, 30^^485, 32^^162, 33^^16, 36^^2679, 40^^410, 42^^453,
%45^^99, 48^^1005, 54^^94, 55^^64, 56^^294, 60^^182, 63^^67, 64^^16, 72^^1236, 
%81^^16, 84^^80, 96^^44, 108^^32 *}

Let $\calA_2'$ be the set of agreeable subgroups $G\subseteq\GL_2(\Zhat)$ such that the following hold:
\begin{itemize}
\item 
$G$ has genus at least $2$ and every agreeable group $G\subsetneq G' \subseteq \GL_2(\Zhat)$ has genus at most $1$,
\item
the level of $G$ is not divisible by any prime $\ell>19$.
\end{itemize}
The set $\calA_2'$ is stable under conjugation by $\GL_2(\Zhat)$.   Let $\calA_2$ be a set of representatives of the $\GL_2(\Zhat)$-conjugacy classes of $\calA_2'$.

\begin{thm} \label{T:main agreeable 2}
\begin{romanenum}
\item \label{T:main agreeable 2 i}
The set $\calA_2$ is finite and computable.
\item \label{T:main agreeable 2 ii}
Take any agreeable subgroup $G$ of $\GL_2(\Zhat)$ with genus at least $2$ that satisfies $G_\ell \supseteq \SL_2(\ZZ_\ell)$ for all $\ell >19$.  Then $G$ is conjugate in $\GL_2(\Zhat)$ to a subgroup of some group in $\calA_2$.
\end{romanenum}
\end{thm}

For each number field $K$, let $J_K$ be the intersection of $K$ with the subset
\[
\bigcup_{G \in \calA_2, \, K_G \subseteq K} \pi_G(X_G(K)) 
\]
of $K\cup\{\infty\}$.   The set $J_K$ is finite by Theorem~\ref{T:main agreeable 2}(\ref{T:main agreeable 2 i}) and Faltings' theorem.  

\begin{thm} \label{T:loose agreeable}
Let $K$ be a number field and let $E/K$ be a non-CM elliptic curve with $j_E\notin J_K$ that satisfies $\rho_{E,\ell}(\Gal_K)\supseteq \SL_2(\ZZ/\ell\ZZ)$ for all primes $\ell >19$.   Take a group $G \in \calA_1$ with maximal index $[\GL_2(\Zhat):G]$ for which $K_G\subseteq K$ and $j_E \in \pi_{G}(X_G(K))$.  Then $G$ and the agreeable closure of $G_E$ are conjugate in $\GL_2(\Zhat)$.
\end{thm}

Since the sets $\calA_1$ and $\calA_2$ are both finite, we can compute a model for the curve $X_G$ and compute the morphism $\pi_G$, with respect to this model, for each group $G\in \calA_1\cup \calA_2$. 

We can effectively determine whether a number in $K$ lies in the set $J_K$ by using explicit models of the curves $X_G$ for $G\in \calA_2$. Using explicit models of $X_G$, with $G\in \calA_1$, Theorem~\ref{T:loose agreeable} lets us compute the agreeable closure of $G_E$, up to conjugacy, for all elliptic curves $E/K$ satisfying the assumptions of the theorem.

\begin{proof}[Proof of Theorem~\ref{T:indices}]
Consider a number field $K$.  Take any non-CM elliptic curve $E/K$ with $j_E \notin J_K$ that satisfies $\rho_{E,\ell}(\Gal_K)\supseteq \SL_2(\ZZ_\ell)$ for $\ell>19$.  By Theorem~\ref{T:loose agreeable}, there is a group $\calG \in \calA_1$ that is conjugate to the agreeable closure of $G_E$.  There is no harm in conjugating $G_E$ so that $G_E\subseteq \calG$.    We have $[\calG,\calG]=[G_E,G_E]$ since $\calG$ is the agreeable closure of $G_E$.   Therefore, $[\calG,\calG] \subseteq G_E\cap \SL_2(\Zhat)$ by Lemma~\ref{L:KW}(\ref{L:KW ii}).   In particular, we have
\[
[\SL_2(\Zhat): G_E\cap \SL_2(\Zhat) ] \leq [\SL_2(\Zhat): [\calG,\calG] ].
\]
The group $\Zhat^\times/\det(\calG)$ is an elementary $2$-group since $\Zhat^\times I \subseteq \calG$.  Therefore, the number field $K_\calG$ is the compositum of quadratic extensions of $\QQ$.   We have $K_\calG \subseteq K$ since $\chi_\cyc(\Gal_K)=\det(G_E) \subseteq \det(\calG)$.

For the finite number of groups $\calG \in \calA_1$, one can compute the index $[\SL_2(\Zhat): [\calG,\calG] ]$ and the number field $K_\calG$; this data can be found in \cite{newgithub}.  All but the last inequality in Theorem~\ref{T:indices} follow from a direct inspection of this data.  

Suppose $K=\QQ$.  Only groups $\calG\in \calA_1$ with $\det(\calG)=\Zhat^\times$ will arise.   After possibly increasing the finite set $J_K$, we need only consider those groups $\calG$ for which $X_\calG(\QQ)$ is infinite.  Since $\calG$ is agreeable, Theorem~\ref{T:main agreeable 1} implies that the level of $\calG$ is not divisible by any prime $\ell>19$.    Therefore, $\calG$ is conjugate in $\GL_2(\Zhat)$ to one of the groups in the finite set $\mathscr{A}$ from Theorem 1.9 of \cite{openimage} with $X_\calG(\QQ)$ infinite (see also Remark~\ref{R:same agreeable}). 
The groups $\calG$ in this finite set $\mathscr{A}$ have been computed and for each group we have also computed $[\SL_2(\Zhat):[\calG,\calG]]$; this data can be found in the repository \cite{github} for the paper \cite{openimage}.  The largest integer that occurs as $[\SL_2(\Zhat):[\calG,\calG]]$, as we vary over all $\calG \in \mathscr{A}$, is $1536$.   Another proof of the bound $1536$ can be found in \cite{possibleindices}.
\end{proof}

\subsection{Some abelian quotients}

Fix an agreeable subgroup $\calG$ of $\GL_2(\Zhat)$.   Consider any non-CM elliptic curve $E$ over a number field $K$ for which the agreeable closure of $G_E$ is $\calG$.   We have $[G_E,G_E]=[\calG,\calG]$ so $G_E$ is a normal subgroup of $\calG$ and $\calG/G_E$ is abelian.   Moreover, $\calG/G_E$ is a finite abelian group since $G_E$ is open in $\calG$ by Theorem~\ref{T:Serre 1972}.  

There may be infinitely many open subgroups $G$ of $\calG$ with $[\calG,\calG]\subseteq G$.   In order to make future computations easier, we will want to work with groups $G$ with small level and small index $[\Zhat^\times:\det(G)]$.  The following theorem, which we prove in \S\ref{SS: nice supply proof}, promises a finite collection of nice subgroups $G\subseteq \calG$ that will be suitable for our applications.

\begin{thm} \label{T:nice supply of G}
Let $N$ be the least common multiple of the levels of $\calG$ and $[\calG,\calG]$.  Then there is a computable finite set $\calS_\calG$ of open subgroups of $\calG$ such that:
\begin{itemize}
\item
For every group $G \in \calS_\calG$, we have $[\calG,\calG] \subseteq G \cap \SL_2(\Zhat) \subseteq \calG \cap \SL_2(\Zhat)$.
\item
For every group $G \in \calS_\calG$, the level of $G$ divides some power of $2$ times $N$. 
\item
For every group $G\in \calS_\calG$, $[\Zhat^\times: \det(G)]$ is a power of $2$.
\item 
For every group $[\calG,\calG] \subseteq W \subseteq \calG \cap \SL_2(\Zhat)$, we have $G \cap \SL_2(\Zhat)=W$ for a unique $G\in \calS_\calG$.
\end{itemize}
\end{thm}

\subsection{Some abelian representations}  \label{SS:etale overview}

Throughout this section we fix an agreeable subgroup $\calG$ of $\GL_2(\Zhat)$ and a finite set of open subgroups $\calS_\calG$ of $\calG$ as in Theorem~\ref{T:nice supply of G}.   Fix an integer $N\geq 3$ that is divisible by the level of $[\calG,\calG]$, the level of $\calG$, and the level of each $G\in \calS_\calG$.

 Define the open subvariety $U_\calG:=\pi_\calG^{-1}(\PP^1_{K_\calG}-\{0,1728,\infty\})$ of $X_\calG$.  The Weierstrass equation
\begin{align} 
\label{E:generic Weierstrass}
y^2 =  x^3 -  27 \cdot j (j-1728)  \cdot  x +54 \cdot j (j-1728)^2,
\end{align}
with $j=\pi_\calG$, defines an elliptic scheme $\scrE_\calG$ over $U_\calG$.  For a number field $K\supseteq K_\calG$ and point $u\in U_\calG(K)$, the fiber of $\scrE_\calG$ over $u$ is the elliptic curve $\scrE_{\calG,u}$ over $K$ given by (\ref{E:generic Weierstrass}) with $j$ replaced by $\pi_\calG(u)\in K-\{0,1728\}$; it has $j$-invariant $\pi_\calG(u)$. 

Let $\bbar{\calG}$ be the image of $\calG$ modulo $N$.  As in \cite[\S6.3.1]{openimage}, we have a surjective and continuous representation
\[
\varrho_{\scrE_\calG,N}^*\colon \pi_1(U_\calG,\bbar\eta) \to \bbar{\calG} \subseteq \GL_2(\ZZ/N\ZZ),
\]
where $\bbar\eta$ is a particular geometric generic point of $U_\calG$ and $\pi_1$ denotes the \'etale fundamental group (the only difference being to base extend by $K_G$ first).  The representation $\varrho_{\scrE}^*$ can be constructed in a similar fashion to our adelic representations for elliptic curves;  the $N$-torsion subscheme $\scrE_\calG[N]$ can be viewed as a lisse sheaf on $U_\calG$ that gives rise to the representation.  For any number field $K\supseteq K_G$ and point $u\in U_\calG(K)$, the specialization of $\varrho_{\scrE_\calG,N}^*$ at $u$ defines a representation $\Gal_K \to \bbar\calG \subseteq \GL_2(\Zhat)$ that is isomorphic to $\rho_{(\scrE_{\calG})_u,N}^*$.  
 
Now take any group $G \in \calS_\calG$ and let $\bbar{G}$ be the image of $G$ modulo $N$.  Since the level of $G$ divides $N$ by our choice of $N$, reduction modulo $N$ induces an isomorphism $\calG/G\xrightarrow{\sim} \bbar\calG/\bbar G$ that we will view as an equality.  Define the homomorphism
\[
\alpha_G\colon \pi_1(U_\calG) \to \calG/G
\]
by composing $\varrho_{\scrE_\calG,N}^*$ with the quotient map $\calG\to \bbar{\calG}/\bbar{G}=\calG/G$; we may suppress the point $\bbar\eta$ since $\calG/G$ is abelian.  

\begin{prop}  \label{P:alpha computable}
The homomorphism $\alpha_G$ is computable, i.e., one can compute a model of $U_\calG$ and an \'etale cover $Y\to U_\calG$ corresponding to $\alpha_G$ along with the action of $\calG/G$ on $Y$.  In particular, for any number field $K\supseteq K_\calG$ and point $u\in U_\calG(K)$, one can compute the specialization $\Gal_K\to \calG/G$ of $\alpha_G$ at $u$.
\end{prop}
\begin{proof}
This follows from the same argument as in \cite[\S11]{openimage} except working over $K_G$.  A slight difference to keep in mind is that the field of constants $K_\calG$ in $K_\calG(U_\calG)$ will not be algebraically closed in the function field of $Y$ when $\det(G)$ is a proper subgroup of $\det(\calG)$; in \cite{openimage}, we only considered cases where $\det(G)=\det(\calG)$.
\end{proof}

\subsubsection{A description of $G_E$}
Now consider any non-CM elliptic curve $E$ defined over a number field $K$ for which the agreeable closure is conjugate to $\calG$ in $\GL_2(\Zhat)$.    

By Proposition~\ref{P:key property}, we have
$j_E = \pi_\calG(u)$ for some point $u\in U_\calG(K)$.   The elliptic curve $E$ is a quadratic twist of $E':=(\scrE_\calG)_u$ by a character $\chi\colon \Gal_K\to \{\pm 1\}$ since $E$ is non-CM and the two elliptic curves have the same $j$-invariant.    For each group $G\in \calS_\calG$, let 
\[
\alpha_{G,E} \colon \Gal_K \to \calG/G
\]
be the homomorphism that is the product of $\chi$ and the specialization of $\alpha_G$ at $u$.

\begin{prop} \label{P:nouveau G choice}
\begin{romanenum}
\item \label{P:nouveau G choice i}
There is a unique group $G\in \calS_\calG$ such that $\alpha_{G,E}(\Gal_{K^\cyc})=1$ and $G\cap \SL_2(\Zhat)$ is minimal with respect to inclusion.   
\item \label{P:nouveau G choice ii}
Take $G\in \calS_\calG$ as in (\ref{P:nouveau G choice i}).  The groups $G_E$ and
\[
\calH_E:= \{ g \in \calG : \det g\, \in \chi_\cyc(\Gal_K),\, g G = \gamma_E(\det g) \}
\]
are conjugate in $\GL_2(\Zhat)$, where $\gamma_E\colon \chi_\cyc(\Gal_K)\to \calG/G$ is the unique homomorphism satisfying $\alpha_{G,E}(\sigma) = \gamma_E(\chi_\cyc(\sigma)^{-1})$ for all $\sigma\in \Gal_K$.
\end{romanenum}
\end{prop}
\begin{proof}
The specialization of $\varrho_{\scrE_\calG,N}^*$ at $u$ is a representation $\Gal_K\to \bbar{\calG} \subseteq \GL_2(\ZZ/N\ZZ)$ isomorphic to $\rho_{(\scrE_\calG)_u,N}^*=\rho_{E',N}^*$.   So by replacing $\rho_{E',N}^*$ with an isomorphic representation, we may assume that it is the specialization of $\varrho_{\scrE_\calG,N}^*$ at $u$.  In particular, we have $\rho_{E',N}^*(\Gal_K)\subseteq \bbar{\calG}$.    We also have $\rho_{E'}^*(\Gal_K) \subseteq \calG$ since the level of $\calG$ divides $N$.  Since $E$ is the quadratic twist of $E'$ by $\chi$, we may assume that $\rho_{E}^* = \chi \cdot \rho_{E'}^*$ and hence also $\rho_{E,N}^* = \chi \cdot \rho_{E',N}^*$.   In particular, $\rho_{E,N}^*(\Gal_K)\subseteq \bbar{\calG}$ and $G_E=\rho_E^*(\Gal_K) \subseteq \calG$ since $-I\in \calG$.  

Now take any $G\in \calS_\calG$.  The homomorphism $\alpha_{G,E}$ agrees with the composition of $\rho_{E,N}^*\colon \Gal_K\to \bbar{\calG}$ with the quotient map $\bbar{\calG}\to \bbar{\calG}/\bbar{G}=\calG/G$.    Therefore, $\alpha_{G,E}(\Gal_{K^\cyc})$ is equal to the image of $\rho_E^*(\Gal_{K^\cyc})=G_E\cap \SL_2(\Zhat)$ in $\calG/G$, where we have used Lemma~\ref{L:KW}(\ref{L:KW i}).   So $\alpha_{G,E}(\Gal_{K^\cyc})=1$ if and only if $G_E \cap \SL_2(\Zhat) \subseteq G \cap \SL_2(\Zhat)$.  Thus to prove (\ref{P:nouveau G choice i}), it suffices to show that $[\calG,\calG] \subseteq G_E \cap \SL_2(\Zhat) \subseteq \calG \cap\SL_2(\Zhat)$ since any such group is of the form $G\cap \SL_2(\Zhat)$ for a unique $G\in \calS_\calG$.   The group $\calG$ is the agreeable closure of $G_E$ since it is conjugate to the agreeable closure and $G_E\subseteq \calG$.  Therefore, $[G_E,G_E]=[\calG,\calG]$   and hence  $[\calG,\calG] \subseteq G_E \cap \SL_2(\Zhat) \subseteq \calG\cap \SL_2(\Zhat)$ by Lemma~\ref{L:KW}(\ref{L:KW ii}).  

We may now suppose that $G$ is chosen as in (\ref{P:nouveau G choice i}).   We have just shown that $G\cap \SL_2(\Zhat)=G_E\cap \SL_2(\Zhat)$.  We have already made choices so that $G_E\subseteq \calG$.   For each $g\in G_E$, we have $\det g \in \det(G_E)=\chi_\cyc(\Gal_K)$.   Note that the existence and uniqueness of $\gamma_E$ is clear since $\chi_\cyc$ induces an isomorphism $\Gal(K^\cyc/K) \xrightarrow{\sim} \chi_\cyc(\Gal_K)$.

We claim that $gG=\gamma_E(\det g)$ for all $g\in G_E$. Take any $\sigma\in \Gal_K$.  From our identification $\bbar\calG/\bbar{G}=\calG/G$, $\rho_{E,N}^*(\sigma)\cdot \bbar{G}$ and $\rho_E^*(\sigma) \cdot G$ represent the same coset.  Therefore,  $\rho_E^*(\sigma) \cdot G = \alpha_{G,E}(\sigma) = \gamma_E(\chi_\cyc(\sigma)^{-1})$.   Since $\det \circ \rho_E^* = \chi_\cyc^{-1}|_{\Gal_K}$, we have $\rho_E^*(\sigma) \cdot G = \gamma_E(\det \rho_E^*(\sigma))$.  The claim follows since $\sigma$ was an arbitrary element of $\Gal_K$.

Using the claim, we have now shown that $G_E\subseteq \calH_E$.   Taking determinants gives $\chi_\cyc(\Gal_K) =\det(G_E) \subseteq \det(\calH_E) \subseteq \chi_\cyc(\Gal_K)$ and hence $\det(G_E) = \det(\calH_E)$.  We also have 
\[
\calH_E \cap \SL_2(\Zhat)= \{  g \in \calG   : g\in \SL_2(\Zhat),\, g G =  G\}=G\cap \SL_2(\Zhat)= G_E \cap \SL_2(\Zhat).
\]
Therefore, $G_E=\calH_E$ since $G_E$ is a subgroup of $\calH_E$ with the same determinant and the same intersection with $\SL_2(\Zhat)$.
\end{proof}

\subsection{Computing the Galois image for most elliptic curves}

For our algorithm, we first shall perform some one-time precomputations.  For each group $\calG\in \calA_1 \cup \calA_2$, we can compute a model for the curve $X_\calG$ and, with respect to this model, compute the morphism $\pi_\calG$.     For each $\calG \in \calA_1$, we can compute a set $\calS_\calG$ as in Theorem~\ref{T:nice supply of G}.   For each $\calG\in \calA_1$ and $G\in \calS_\calG$, we can compute $\alpha_G$ as in Proposition~\ref{P:alpha computable}.

Fix an explicit non-CM elliptic curve $E$ defined over a number field $K$ for which $j_E \notin J_K$ and $\rho_{E,\ell}(\Gal_K)\supseteq \SL_2(\ZZ/\ell\ZZ)$ for all primes $\ell>19$.  Note that the condition $j_E\notin J_K$ can be checked since the morphisms $\pi_\calG$ with $\calG\in \calA_2$ are computed already.  We will discuss the conditions on the $\rho_{E,\ell}$ in \S\ref{SS:loose ends 1}.  We can also compute the open group $\chi_\cyc(\Gal_K)\subseteq \Zhat^\times$.

Using Theorem~\ref{T:loose agreeable} and our precomputed modular curves, we can find a group $\calG\in \calA_1$ that is conjugate in $\GL_2(\Zhat)$ to the agreeable closure of $G_E=\rho^*_E(\Gal_K)$.   Choose a point $u\in U_\calG(K)$ for which $\pi_\calG(u)=j_E$.   Let $E'$ be the elliptic curve over $K$ defined by the equation (\ref{E:generic Weierstrass}) with $j$ replaced by $j_E$.    The curve $E$ is a quadratic twist of $E'$ by a computable character $\chi\colon \Gal_K \to \{\pm 1\}$ since $E$ is non-CM and $j_{E'}=j_E$.

Take any group $G\in \calS_\calG$.   We define $\alpha_{E,G}\colon \Gal_K\to \calG/G$ to be product of $\chi$ and the specialization of $\alpha_G$ at $u$; this is computable by using our precomputed $\alpha_G$.   We can find the group $G\in \calS_\calG$ that satisfies Proposition~\ref{P:nouveau G choice}(\ref{P:nouveau G choice i}); for the rest of the section, we work with this fixed group $G$.    There is a unique computable homomorphism $\gamma_E\colon \chi_\cyc(\Gal_K)\to \calG/G$ satisfying $\alpha_{G,E}(\sigma) = \gamma_E(\chi_\cyc(\sigma)^{-1})$ for all $\sigma\in \Gal_K$.

From $\calG$, $G$, $\chi_\cyc(\Gal_K)$ and $\gamma_E$, Proposition~\ref{P:nouveau G choice}(\ref{P:nouveau G choice ii}) gives an explicit subgroup $\calH_E$ of $\GL_2(\Zhat)$ that is conjugate to $G_E$.   This is the desired explicit computation of $G_E$ up to conjugacy.

\subsection{Loose ends 1: images modulo $\ell$ and uniformity} \label{SS:loose ends 1}

Consider a non-CM elliptic curve $E$ over a number field $K$.  A consequence of Theorem~\ref{T:Serre 1972}, and also one of the ingredients of its proof, is that $\rho_{E,\ell}(\Gal_K) \supseteq \SL_2(\ZZ/\ell\ZZ)$ for all primes $\ell > c_{E,K}$, where $c_{E,K}$ is a positive integer which we take to be minimal.  In the case $K=\QQ$, Serre asked whether $c_{E,\QQ}$ can be bounded independent of $E$, see~\cite[\S4.3]{Serre-Inv72} and the final remarks of \cite{MR644559} where he asks if $c_{E,\QQ}\leq 37$.   We formulate this as a conjecture over a general number field.

\begin{conj}[Serre uniformity problem] \label{C:uniformity}
For any number field $K$, the following equivalent conditions hold:
\begin{alphenum} 
\item \label{C:uniformity a}
There is a constant $c_K$ such that for any prime $\ell>c_K$ and any non-CM elliptic curves $E/K$, we have $\rho_{E,\ell}(\Gal_K)\supseteq \SL_2(\ZZ/\ell\ZZ)$.   
\item \label{C:uniformity b}
There is a finite set $\calJ_K \subseteq K$ such that for any prime $\ell>19$ and any non-CM elliptic curve $E/K$ with $j_E\notin \calJ_K$, we have $\rho_{E,\ell}(\Gal_K)\supseteq \SL_2(\ZZ/\ell\ZZ)$.
\end{alphenum}
\end{conj}

It is an important problem to determine the finite set of primes $\ell>19$ for which $\rho_{E,\ell}(\Gal_K)\not\supseteq \SL_2(\ZZ/\ell\ZZ)$.  For a fixed prime $\ell>19$, there are fast probabilistic methods of Sutherland \cite{MR3482279} to identify the image of $\rho_{E,\ell}$ up to a notion of local conjugacy.  Note that whenever the algorithm of \cite{MR3482279} predicts $\rho_{E,\ell}(\Gal_K)\supseteq \SL_2(\ZZ/\ell\ZZ)$, the result is guaranteed to be correct.   
 
 There are various bounds for $c_{E,K}$ in the literature.  For example, in \cite{MR1998390} one finds an explicit upper bound for $c_{E,K}$;  however, it is too large for use in practice.     Bounds for $c_{E,K}$ assuming GRH, like suggested in \cite{MR3161774}, should do better.
 
 In the case $K=\QQ$, \cite{surjectivityalgorithm} gives an efficient algorithm that computes a relatively small finite set of primes $S$ for which $\rho_{E,\ell}(\Gal_\QQ)\supseteq\SL_2(\ZZ/\ell\ZZ)$ for all $\ell>19$ with $\ell\notin S$ (one can then quickly address any primes in $S$).   An analogous algorithm over a general number field should be worked out.

\begin{proof}[Proof of the equivalence in Conjecture~\ref{C:uniformity}]
Take any prime $\ell>19$ and any non-CM elliptic curve $E/K$.  Since $\SL_2(\ZZ/\ell\ZZ)$ is equal to its own commutator subgroup, see Lemma~\ref{L:comm}(\ref{L:comm ii}),  we have $\rho_{E,\ell}(\Gal_K)\supseteq \SL_2(\ZZ/\ell\ZZ)$ if and only if $\pm \rho_{E,\ell}(\Gal_K)\supseteq \SL_2(\ZZ/\ell\ZZ)$.   Since $\pm \rho_{E,\ell}(\Gal_K)$, up to conjugacy, does not change if we replace $E/K$ by a quadratic twist and $E$ is non-CM,   we find that the condition $\rho_{E,\ell}(\Gal_K)\supseteq \SL_2(\ZZ/\ell\ZZ)$ depends only on $j_E$, $K$ and $\ell$.    Therefore, condition (\ref{C:uniformity b}) implies (\ref{C:uniformity a}) since for each $j\in \calJ_K$ that is the $j$-invariant of a non-CM elliptic curve $E/K$, we can conclude by Theorem~\ref{T:Serre 1972}.

We now assume that (\ref{C:uniformity a}) holds for some constant $c_K$.   Suppose that we have $\rho_{E,\ell}(\Gal_K) \not \supseteq \SL_2(\ZZ/\ell\ZZ)$ for some non-CM elliptic curve $E/K$ and prime $\ell>19$.   We have $19<\ell <c_K$.   Let $\calG$ be the open subgroup of $\GL_2(\Zhat)$ of level $\ell$ whose image modulo $\ell$ is the transpose of $\pm \rho_{E,\ell}(\Gal_K)$ (and hence does not contain $\SL_2(\ZZ/\ell\ZZ)$).    Using the classification in \cite{MR2016709} and $\ell>19$, the genus of $X_\calG$ is at least $2$. We have $j_E \in \pi_\calG(X_\calG(K))$ by Proposition~\ref{P:key property}.   Therefore, (\ref{C:uniformity b}) holds since only finitely many groups $\calG$ arise and $X_\calG(K)$ is finite by Faltings' theorem.
\end{proof}

\subsection{Loose ends 2: exceptional images}

In general, computing $G_E$ for an arbitrary non-CM elliptic curve $E$ over a number field is still an extremely difficult problem.   The fundamental reason being that \emph{every} open subgroup of $\GL_2(\Zhat)$ will occur as such an image (to prove this one need only show that $\GL_2(\Zhat)$ occurs, see \cite{MR2721742}).

For a fixed number field $K$, and assuming Conjecture~\ref{C:uniformity} for simplicity, we have shown how to compute $G_E$ for all non-CM elliptic curves $E/K$ whose $j$-invariant lies away from some finite subset of $K$.   What makes this proposed algorithm especially practical is that one need only compute a finite number of modular curves and this can be done ahead of time.   

For non-CM elliptic curves over $K$ with one of the excluded $j$-invariants, a similar approach works but requires more modular curves computations or ad hoc computations.  For how we dealt with this in the $K=\QQ$ case, see \S10.2 and \S12.3 of \cite{openimage}.

\section{Basic group theory}

In this section, we collect some basic group theory facts that will be used in our arguments.

\subsection{Goursat's lemma}

\begin{lemma}[Goursat's lemma, \cite{Ribet-76}*{Lemma 5.2.1}]  
\label{L:Goursat}
Let $G_1$ and $G_2$ be two groups and let $H$ be a subgroup of $G_1\times G_2$ so that the projection maps $p_1\colon H\to G_1$ and $p_2\colon H \to G_2$ are surjective.   Let $B_1$ and $B_2$ be the normal subgroups of $G_1$ and $G_2$, respectively, for which $\ker(p_2)=B_1\times \{1\}$ and $\ker(p_1)=\{1\}\times B_2$.    Then the image of $H$ in $(G_1\times G_2)/(B_1\times B_2)= G_1/B_1\times G_2/B_2$ is the graph of an isomorphism $G_1/B_1\xrightarrow{\sim} G_2/B_2$.
\end{lemma}

\subsection{Commutator subgroups}

\begin{lemma} \cite[Lemma 7.7]{openimage} \label{L:comm}
\begin{romanenum}
\item \label{L:comm ii}
The commutator subgroup of $\SL_2(\ZZ_\ell)$ is equal to $\SL_2(\ZZ_\ell)$ for all $\ell > 3$.
\item \label{L:comm iii}
The commutator subgroup of $\GL_2(\ZZ_\ell)$ is equal to $\SL_2(\ZZ_\ell)$ for all $\ell \geq 3$.
\item  \label{L:comm iv}
The commutator subgroup of $\SL_2(\ZZ_3)$ has level $3$ and index $3$.
\end{romanenum}
\end{lemma}

\begin{lemma} \cite[Lemma 7.10]{openimage} \label{L:openness of commutator}
Let $G$ be an open subgroup of $\GL_2(\Zhat)$ or $\SL_2(\Zhat)$.  Then the commutator subgroup $[G,G]$ is an open subgroup of $\SL_2(\Zhat)$.
\end{lemma}

\begin{lemma} \label{L:Well properties}
Take any prime $\ell \geq 3$.  
\begin{romanenum}
\item
There is a unique closed normal subgroup $W_\ell$ of $\SL_2(\ZZ_\ell)$ for which $\SL_2(\ZZ_\ell)/W_\ell$ is a simple group.  
\item
Suppose $\ell > 3$.  Then the group $W_\ell$ consists of the matrices in $\SL_2(\ZZ_\ell)$ whose image modulo $\ell$ are $\pm I$.   We have $\SL_2(\ZZ_\ell)/W_\ell \cong \SL_2(\FF_\ell)/\{\pm I\}$.
\item
The group $W_3$ is the commutator subgroup of $\SL_2(\ZZ_3)$ and $\SL_2(\ZZ_3)/W_3$ is cyclic of order $3$.  
\end{romanenum}
\end{lemma}
\begin{proof}
Let $Q$ be a finite simple group that is a quotient of $\SL_2(\ZZ_\ell)$ by a closed normal subgroup.

Suppose $\ell>3$.     The simple group $Q$ is nonabelian by Lemma~\ref{L:comm}(\ref{L:comm ii}).  Since pro-$\ell$ groups are prosolvable and $Q$ is simple and nonabelian, we find that any continuous surjective homomorphism $\SL_2(\ZZ_\ell)\twoheadrightarrow Q$ factors through $\SL_2(\ZZ/\ell\ZZ)/\{\pm I\} \twoheadrightarrow Q$.  The lemma is immediate in this case since $\SL_2(\ZZ/\ell\ZZ)/\{\pm I\}$ is simple.

The group $\SL_2(\ZZ_3)$ is prosolvable since pro-$3$ groups are prosolvable and $\SL_2(\ZZ/3\ZZ)$ is solvable.  Therefore, $Q$ is a cyclic group of prime order.  The lemma for $\ell=3$ follows since the quotient group $\SL_2(\ZZ_3)/[\SL_2(\ZZ_3),\SL_2(\ZZ_3)]$ is isomorphic to $\ZZ/3\ZZ$ by Lemma~\ref{L:comm}(\ref{L:comm iv}).  
\end{proof}

\begin{lemma} \label{L:old agreeable}
Let $G$ be an open subgroup of $\GL_2(\Zhat)$ or $\SL_2(\Zhat)$.   Take any prime $\ell>5$.  Then $G_\ell \supseteq \SL_2(\ZZ_\ell)$ if and only if $\ell$ does not divide the level of $[G,G]$.  %\subseteq \SL_2(\Zhat)$.
\end{lemma}
\begin{proof}
First suppose that $\ell$ does not divide the level of $[G,G]\subseteq \SL_2(\Zhat)$. Then $G_\ell \supseteq [G,G]_\ell \supseteq \SL_2(\ZZ_\ell)$.

Now suppose that $G_\ell \supseteq \SL_2(\ZZ_\ell)$.  Define $G'=[G,G]$; it is an open subgroup of $\SL_2(\Zhat)$ by Lemma~\ref{L:openness of commutator}.   We have $G'_\ell=[G_\ell,G_\ell] \supseteq \SL_2(\ZZ_\ell)$ by Lemma~\ref{L:comm}(\ref{L:comm ii}).   The level of $[G',G']$ divides the level of the larger group $[G,G]$.   So after replacing $G$ by $G'$, we may assume that $G\subseteq \SL_2(\Zhat)$ and that $G_\ell=\SL_2(\ZZ_\ell)$.   Let $H$ be the image of $G$ under the projection map to $\prod_{p\neq \ell} \SL_2(\ZZ_p)$.  We may view $G$ as a subgroup of $H \times \SL_2(\ZZ_\ell)$ for which the projections to the factors $H$ and $\SL_2(\ZZ_\ell)$ are surjective.   By Goursat's lemma (Lemma~\ref{L:Goursat}), we have $B_1\times B_2 \subseteq G$ and $H/B_1\cong \SL_2(\ZZ_\ell)/B_2$, where $B_1$ and $B_2$ are certain normal subgroup of $H$ and $\SL_2(\ZZ_\ell)$, respectively.  In our case, the groups $B_1$ and $B_2$ are also closed.   

Suppose that $B_2\neq \SL_2(\ZZ_\ell)$.  By Lemma~\ref{L:Well properties}, the simple group $Q:=\SL_2(\FF_\ell)/\{\pm I\}$ is isomorphic to a quotient of $\SL_2(\ZZ_\ell)/B_2\cong H/B_1$.   Therefore, $Q$ is a quotient of $H_p$ for some prime $p\neq \ell$, where $H_p$ is a closed subgroup of $\GL_2(\ZZ_p)$.   The group $Q$ is not isomorphic to either of the groups $\SL_2(\FF_p)/\{\pm I\}$ or $A_5$ by cardinality considerations.    However, this contradicts the computation of the sets ``$\operatorname{Occ}(\GL_2(\ZZ_\ell))$'' in \cite[IV \S3.4]{Serre-abelian}.   

Therefore, $B_2=\SL_2(\ZZ_\ell)$ and hence also $B_1=H$.  From the inclusions $H\times \SL_2(\ZZ_\ell) \supseteq G \supseteq B_1 \times B_2$, we deduce that $G=H\times \SL_2(\ZZ_\ell)$.  Therefore, $[G,G]=[H,H] \times \SL_2(\ZZ_\ell)$ by Lemma~\ref{L:comm}(\ref{L:comm ii}) and hence $\ell$ does not divide the level of $[G,G]$.
\end{proof}

\begin{lemma} \label{L:Serre mod ell to ell-adic}
Fix a prime $\ell\geq 5$ and let $G$ be a closed subgroup of $\GL_2(\ZZ_\ell)$.  Then $G\supseteq \SL_2(\ZZ_\ell)$ if and only if the image of $G$ modulo $\ell$ contains $\SL_2(\ZZ/\ell\ZZ)$.  
\end{lemma}
\begin{proof}
After replacing $G$ by $[G,G]$, and using Lemma~\ref{L:comm}(\ref{L:comm ii}), we may assume that $G\subseteq \SL_2(\ZZ_\ell)$.  The lemma now follows from \cite[IV \S3.4 Lemma 3]{Serre-abelian}.
\end{proof}

\subsection{Determining the level of groups}

The following lemmas give cases where we can show that a subgroup of $\GL_2(\ZZ_N)$ is open and also give a bound on its level.

\begin{lemma} \cite[Lemma 7.6]{openimage} \label{L:level by adjoining scalars}
Fix an integer $N>1$ with $N\not\equiv 2 \pmod{4}$.  Let $G$ be a subgroup of $\GL_2(\ZZ_N)$ for which $G \cap \SL_2(\ZZ_N)$ is an open subgroup of $\SL_2(\ZZ_N)$ whose level divides $N$. Define $N_1:=N$ if $N$ is odd and $N_1:=2N$ if $N$ is even.   Then $\ZZ_N^\times \cdot G$ is an open subgroup of $\GL_2(\ZZ_N)$ whose level divides $N_1$.
\end{lemma}

\begin{lemma} \label{L:how to look for maximal}
Fix an integer $N>1$ with $N\not\equiv 2 \pmod{4}$.   For each prime $\ell$ dividing  $N$, define the integer
\[
N_\ell := \ell^{e_\ell} \prod_{{p|N,\, p^2\equiv 1 \bmod{\ell}}} p,
\]
where $\ell^{e_\ell}$ is the largest power of $\ell$ dividing $N$.  Note that $N_\ell$ is a divisor of $N$.

Let $\calG$ be an open subgroup of $\GL_2(\ZZ_N)$ whose level divides $N$.   Let $G$ be a maximal open subgroup of $\calG$ whose level does not divide $N$.   Then for some prime $\ell |N$, the images of $\calG$ and $G$ modulo $N_\ell \ell$ are distinct subgroups of $\GL_2(\ZZ/N_\ell \ell \ZZ)$.
\end{lemma}
\begin{proof}
Suppose that $G$ is a maximal open subgroup of $\calG$ such that $G$ and $\calG$ have the same image in $\GL_2(\ZZ/N_\ell \ell\ZZ)$ for all primes $\ell|N$.  Take any $\ell|N$.  Since the level of $\calG$ divides $N$, we find that the image of $G$ in $\GL_2(\ZZ/N_\ell \ell \ZZ)$ contains all the matrices that are congruent to $I$ modulo $N_\ell$.  By \cite[Lemma 7.2]{openimage}, we deduce that $G$ has level dividing $N$.
\end{proof}

\section{Agreeable groups} \label{S:agreeable}

\subsection{Agreeable groups} \label{SS:agreeable basics}

Recall that a subgroup $G$ of $\GL_2(\Zhat)$ is \defi{agreeable} if it is open, contains all the scalar matrices in $\GL_2(\Zhat)$, and each prime dividing the level of $G$ also divides the level of $[G,G]$.

Let $G$ be an open subgroup of $\GL_2(\Zhat)$ and let $M$ be the product of the primes that divide the level of $[G,G]\subseteq \SL_2(\Zhat)$.    We define the \defi{agreeable closure} of $G$ to be the group 
\begin{align}\label{E:agreeable closure} 
\calG:= (\ZZ_M^\times  G_M) \times {\prod}_{\ell\nmid M} \GL_2(\ZZ_\ell).
\end{align}
We now give some basic properties of $\calG$.

\begin{lemma} \label{L:agreeable commutators agree}
\begin{romanenum}
\item \label{L:agreeable commutators agree i}
We have $G\subseteq \calG$ and $[G,G]=[\calG,\calG]$.
\item \label{L:agreeable commutators agree ii}
The group $\calG$ is the minimal agreeable subgroup of $\GL_2(\Zhat)$ that contains $G$.  In particular, $G$ is agreeable if and only if $\calG=G$.
\item \label{L:agreeable commutators agree iv}
If $M'$ is the level of $[G,G] \subseteq \GL_2(\Zhat)$, then the level of $\calG$ divides $2 \operatorname{lcm}(M',4)$.
\end{romanenum}
\end{lemma}
\begin{proof}
The inclusion $G\subseteq \calG$ is clear since $G_M \subseteq \ZZ_M^\times  G_M$.   The integer $M$ is even since the commutator subgroup of $\GL_2(\Zhat)$, and hence also of $G$, has level divisible by $2$.   Since $M$ is even, we have $[\GL_2(\ZZ_\ell),\GL_2(\ZZ_\ell)]=\SL_2(\ZZ_\ell)$ for all $\ell\nmid M$, cf.~Lemma~\ref{L:comm}(\ref{L:comm iii}).   Therefore, $[\calG,\calG] = [G_M,G_M] \times \prod_{\ell\nmid M} \SL_2(\ZZ_\ell) = [G,G]_M \times \prod_{\ell\nmid M} \SL_2(\ZZ_\ell) = [G,G]$, where the last equality uses that $M$ has the same prime divisors as the level of $[G,G]$.   This proves (\ref{L:agreeable commutators agree i}).

Since $[G,G]=[\calG,\calG]$, the integer $M$ is also the product of the primes dividing the level of $[\calG,\calG]$.  From the definition of the group $\calG$, it contains the scalars of $\GL_2(\Zhat)$ and each prime dividing the level of $\calG$ must divide $M$.  Therefore, $\calG$ is agreeable.

Take any agreeable subgroup $B$ of $\GL_2(\Zhat)$ with $G\subseteq B$.   Let $N$ be the product of the primes that divide the level of $[B,B]$.   We have $[G,G] \subseteq [B,B]$, so $N$ divides $M$.  Since $B$ is agreeable and $N|M$, we have $B=B_M \times \prod_{\ell\nmid M} \GL_2(\ZZ_\ell)$.  We have $\calG_M=\ZZ_M^\times G_M  \subseteq B_M$ since $B$ contains the scalars and $G\subseteq B$.  Therefore, $\calG\subseteq B$.   Part (\ref{L:agreeable commutators agree ii}) now follows.

The level of $H:=G \cap \SL_2(\Zhat)$ divides $M'$ since $H\supseteq [G,G]$.  Note that $M$ and $M'$ have the same prime divisors.  Lemma~\ref{L:level by adjoining scalars} implies that $\ZZ_M^\times  H_M$ is an open subgroup of $\GL_2(\ZZ_M)$ whose level divides $2 \operatorname{lcm}(M',4)$.   Part (\ref{L:agreeable commutators agree iv}) follows since $\calG$ contains $(\ZZ_M^\times  H_M) \times \prod_{\ell\nmid M} \GL_2(\ZZ_\ell)$.
\end{proof}

\begin{remark} \label{R:same agreeable}
In \cite{openimage}, we gave a different definition of an \emph{agreeable} subgroup $G$ that insisted on the extra assumption $\det(G)=\Zhat^\times$.   Consider any open subgroup $G$ of $\GL_2(\Zhat)$ for which $\det(G)=\Zhat^\times$.  In the notation of \cite{openimage}, the group $G$ is agreeable if and only $G$ equals (\ref{E:agreeable closure}), cf.~\S8.3 of \cite{openimage} where the agreeable closure is constructed.   In particular, the notions of \emph{agreeable} in this work and in \cite{openimage} are the same for groups with full determinant.
\end{remark}

\subsection{Maximal agreeable subgroups} \label{SS:maximal agreeable}

Fix an agreeable subgroup $\calG$ of $\GL_2(\Zhat)$. Let $M$ be the product of the primes that divide the level of $[\calG,\calG]$.   In this section, we shall describe the maximal agreeable (proper) subgroups of $\calG$.  We start by giving some obvious maximal agreeable subgroups.

\begin{lemma} \label{L:obvious maximal}
\begin{romanenum}
\item \label{L:obvious maximal i}
Let $B$ be a maximal (proper) open subgroup of $\calG_M$ satisfying $B\supseteq \ZZ_M^\times  I$.  Then $G:=B\times\prod_{\ell\nmid M} \GL_2(\ZZ_\ell)$ is a maximal agreeable subgroup of $\calG$.
\item \label{L:obvious maximal ii}
For a prime $p\nmid M$, let $B$ be a maximal (proper) open subgroup of $\GL_2(\ZZ_p)$ that satisfies $B\supseteq \ZZ_p^\times I$ and $B\not\supseteq \SL_2(\ZZ_p)$.  Then $G:=\calG_M\times B \times \prod_{\ell \nmid pM} \GL_2(\ZZ_\ell)$ is a maximal agreeable subgroup of $\calG$.
\item \label{L:obvious maximal iii}
If $3\nmid M$, then $G:=\calG_M \times (\ZZ_3^\times  \SL_2(\ZZ_3)) \times \prod_{\ell\nmid 3M} \GL_2(\ZZ_\ell)$ is a maximal agreeable subgroup of $\calG$.
\end{romanenum}
\end{lemma}
\begin{proof}
Since $\calG$ is agreeable, we have $\calG=\calG_M \times \prod_{\ell\nmid M}\GL_2(\ZZ_\ell)$.  In all the cases, the group $G$ is open, contains $\Zhat^\times I$ and satisfies $G\subseteq \calG$.   Let $N$ be the product of the primes dividing the level of $[G,G] \subseteq \SL_2(\Zhat)$.   We have $M|N$ since $G\subseteq \calG$.   We have $N=M$, $N=pM$ and $N=3M$ in parts (\ref{L:obvious maximal i}), (\ref{L:obvious maximal ii}) and (\ref{L:obvious maximal iii}), respectively; in part (\ref{L:obvious maximal iii}), we use Lemma~\ref{L:comm}(\ref{L:comm iv}).   In all the cases, every prime dividing the level of $G$ also divides $N$.  Thus $G$ is agreeable.     In all the cases, one readily sees that $G$ is a maximal subgroup of $\calG$.
\end{proof}

After setting some more notations, we will describe the maximal agreeable subgroups of $\calG$ that are not covered by Lemma~\ref{L:obvious maximal}.

Fix a prime $p\in \{3,5\}$.   Let $\mathfrak{S}_p$ and $\mathfrak{A}_p$ be the symmetric and alternating groups, respectively, on $p$ letters.  By Lemma~\ref{L:Well properties}, there is a unique closed normal subgroup $W_p$ of $\SL_2(\ZZ_p)$ for which $\SL_2(\ZZ_p)/W_p$ is a finite simple group.   The group $\SL_2(\ZZ_p)/W_p$ is isomorphic to $\mathfrak{A}_p$ (recall the exceptional isomorphism $\PSL_2(\FF_5)/\{\pm I\}\cong \mathfrak{A}_5$).  The group $W_p$ is also normal in $\GL_2(\ZZ_p)$.  Let
\[
\psi_p\colon \GL_2(\ZZ_p) \to \GL_2(\ZZ_p)/(\ZZ_p^\times W_p) \cong \mathfrak{S}_p
\]
be the homomorphism obtained by composing the quotient map with a choice of isomorphism (the existence of this isomorphism is a direct computation and requires $p\in \{3,5\}$).  The group $\psi_p(\SL_2(\ZZ_p))$ is equal to the alternating group $\mathfrak{A}_p$.   Using the uniqueness of $W_p$, we find that a closed subgroup $B$ of $\GL_2(\ZZ_p)$ satisfies $B\supseteq \SL_2(\ZZ_p)$ if and only if $\psi_p(B)\supseteq \mathfrak{A}_p$.

\begin{lemma} \label{L:maximal leftovers}
Let $G$ be a maximal agreeable subgroup of $\calG$ that is not one of the groups described in Lemma~\ref{L:obvious maximal}.
\begin{romanenum}
\item \label{L:maximal leftovers i}
There is a unique prime $p\in \{3,5\}$ such that $p\nmid M$ and $p$ divides the level of $[G,G]$.   We have $G=G_{Mp} \times \prod_{\ell \nmid Mp} \GL_2(\ZZ_\ell)$.
\item \label{L:maximal leftovers ii}
If $p=3$, then $G_p=\GL_2(\ZZ_p)$.   If $p=5$, then $G_p=\ZZ_p^\times \SL_2(\ZZ_p)$.

\item \label{L:maximal leftovers iii}
There is a homomorphism $\varphi\colon \calG_M \to \mathfrak{S}_p$ such that $\varphi(\calG_M)=\psi_p(G_p)$ and 
\[
G_{Mp}=\{(g_1,g_2)\in \calG_M\times G_p : \varphi(g_1)=\psi_p(g_2) \}.
\]
\end{romanenum}
\end{lemma}

\begin{proof}
Since $G$ is a maximal agreeable subgroup of $\calG$, our assumption that $G$ is not one of the groups from Lemma~\ref{L:obvious maximal} implies that the following hold:
\begin{itemize}
\item 
$G_M=\calG_M$,
\item
$G_\ell \supseteq \SL_2(\ZZ_\ell)$ for all primes $\ell \nmid M$,
\item
if $3\nmid M$, then $G_3=\GL_2(\ZZ_3)$.
\end{itemize}
Since $M$ is even and $G\supseteq \Zhat^\times I$, we have $\ZZ_\ell^\times \SL_2(\ZZ_\ell)\subseteq G_\ell \subseteq \GL_2(\ZZ_\ell)$ for all $\ell\nmid M$.  In particular, (\ref{L:maximal leftovers ii}) will follow once we prove (\ref{L:maximal leftovers i}).

Since $G_\ell \supseteq \SL_2(\ZZ_\ell)$ for all $\ell\nmid M$, Lemma~\ref{L:old agreeable} implies that the level of $[G,G]$ is not divisible by any prime $\ell \nmid M$ with $\ell>5$.   Since $M$ is even and $[G,G]\subseteq [\calG,\calG]$, we deduce that the product of the primes dividing the level of $[G,G]$ is $Mm$ for a unique $m | 15$.  We have $m>1$ since $G$ is a proper subgroup of $\calG$ and $G_M=\calG_M$.   Let $p\in \{3,5\}$ be the largest prime dividing $m$.   

We can view $G_{Mp}$ as a subgroup of $G_M \times G_p$.  The projection homomorphisms $\varphi_1 \colon G_{Mp} \to G_M$ and $\varphi_2\colon G_{Mp}\to G_p$ are surjective.  Let $B_1$ and $B_2$ be the normal subgroups of $G_M$ and $G_p$, respectively, for which $\ker(\varphi_2)=B_1\times \{I\}$ and $\ker(\varphi_1)=\{I\}\times B_2$.  Note that $G_{Mp}$ contains $B_1\times B_2$.   We have $\ZZ_M^\times I \subseteq B_1$ and $\ZZ_p^\times I \subseteq B_2$ since $G$ contains all the scalars of $\GL_2(\Zhat)$.  By Goursat's lemma (Lemma~\ref{L:Goursat}), the image of $G_{Mp}$ in $(G_M\times G_p)/(B_1\times B_2)= G_M/B_1 \times G_p/B_2$ is the graph of a group isomorphism $f\colon G_M/B_1\xrightarrow{\sim} G_p/B_2$.

\noindent $\bullet$ First consider the case where $B_2\not \supseteq \SL_2(\ZZ_p)$ and hence $\psi_p(B_2)\not\supseteq \mathfrak{A}_p$.   In particular, $\psi_p(B_2)$ is a normal subgroup of $\psi_p(G_p) \in \{\mathfrak{A}_p,\mathfrak{S}_p\}$ that does not contain $\mathfrak{A}_p$.  Therefore, $\psi_p(B_2)=1$; equivalently, $B_2\subseteq \ZZ_p^\times W_p$.  Define the homomorphism 
\[
\varphi\colon G_M\to G_M/B_1 \xrightarrow{f} G_p/B_2 \to G_p/(\ZZ_p^\times W_p) \hookrightarrow \GL_2(\ZZ_p)/(\ZZ_p^\times W_p) \xrightarrow{\sim} \mathfrak{S}_p,
\]
where the last isomorphism with $\mathfrak{S}_p$ is the same as that in our definition of $\psi_p$.  We have $\varphi(G_M)=\psi_p(G_p) \supseteq \mathfrak{A}_p$ and inclusions
\begin{align*}
G_{Mp}&=\{ (g_1,g_2) \in G_M \times G_p : f(g_1 B_1)=g_2 B_2 \}\\
&\subseteq \{ (g_1,g_2)\in G_M \times G_p : \varphi(g_1)=\psi_p(g_2) \}=:C.
\end{align*}
Define the open subgroup $G':=C\times \prod_{\ell\nmid Mp} \GL_2(\ZZ_\ell)$ of $\GL_2(\Zhat)$.

We claim that $G'$ is agreeable and satisfies $G\subseteq G' \subsetneq \calG$.   We certainly have $G\subseteq G'$ since $G_{Mp}\subseteq C$.  We have $G'\subsetneq \calG$ since $C \subsetneq G_M \times G_p \subseteq \calG_M \times \GL_2(\ZZ_\p) = \calG_{Mp}$, where the strict inclusion uses that $\psi_p$ is non-trivial on $G_p$ and the last equality uses that $p$ does not divide the level of $\calG$.  The inclusion $G'\subseteq \calG$ implies that $[G',G']\subseteq [\calG,\calG]$ and hence the level of $[G',G']$ is divisible by every prime dividing $M$.   Since $M$ is even, we have $[G',G']=[C,C] \times \prod_{\ell\nmid Mp} \SL_2(\ZZ_\ell)$ by Lemma~\ref{L:comm}(\ref{L:comm iii}).    So the level of $[G',G']$ is not divisible by any primes $\ell\nmid Mp$.    Using that $\psi_p(G_p\cap \SL_2(\ZZ_p))=\psi_p(\SL_2(\ZZ_p))=\mathfrak{A}_p$, one finds that the level of $C \cap \SL_2(\ZZ_{Mp})$, and hence also of $[C,C]$, is divisible by $p$.  Combining everything together, we deduce that the product of primes that divide the level of $[G',G']$ is $Mp$.   Observe that the level of $G'$ is divisible only by primes dividing $Mp$.   Since $G$ contains the scalars of $\GL_2(\Zhat)$, we have $\ZZ_{Mp}^\times = \ZZ_M^\times \times \ZZ_p^\times \subseteq B_1\times B_2 \subseteq G_{Mp} \subseteq C$.  Therefore, $G'$ contains all the scalars of $\GL_2(\Zhat)$.  We have now verified that $G'$ is agreeable.

Since $G$ is a maximal agreeable subgroup of $\calG$, the previous claim implies that $G'=G$.  Part (\ref{L:maximal leftovers iii}) with our prime $p$ holds in this case from our definition of $G'$.

\noindent $\bullet$  Now consider the case where $B_2\supseteq \SL_2(\ZZ_p)$.  We will prove that this case cannot occur.

We claim that $[G_{Mp},G_{Mp}]=[G_M,G_M] \times \SL_2(\ZZ_p)$.  It suffices to show that $[G_{Mp},G_{Mp}]\supseteq \{I\}\times \SL_2(\ZZ_p)$.   Take any $g_1,g_2 \in G_p$ with $\det(g_1)=\det(g_2)$.  Since $B_2\supseteq \SL_2(\ZZ_p)$,  $g_1$ and $g_2$ lie in the same coset of $G_p/B_2$.  So there is an $a\in G_M$ such that $(a,g_1)$ and $(a,g_2)$ both lies in $G_{Mp}$.   Taking the commutator of these elements, we find that $(I, g_1 g_2 g_1^{-1} g_2^{-1})$ lies in $[G_{Mp},G_{Mp}]$.    Therefore, $[G_{Mp},G_{Mp}]\supseteq\{I\}\times C$, where $C\subseteq \SL_2(\ZZ_p)$ is the closed group generated by the set $\{ g_1g_2 g_1^{-1} g_2^{-1} : g_1,g_2 \in G_p, \det(g_1)=\det(g_2)\}$.  It thus suffices to show that $C=\SL_2(\ZZ_p)$.   When $p=5$, we have $C=\SL_2(\ZZ_p)$ by Lemma~\ref{L:comm}(\ref{L:comm ii}).  So assume that $p=3$.   Since $C$ contains the commutator subgroup of $\SL_2(\ZZ_3)$, the group $C$ has level $1$ or $3$ by Lemma~\ref{L:comm}(\ref{L:comm iv}).  A simple computation shows that the image of $C$ modulo $3$ is $\SL_2(\ZZ/3\ZZ)$ and hence $C=\SL_2(\ZZ_3)$.  This completes the proof of the claim.

Suppose that $m=p$.  The product of the primes dividing the level of $[G,G]$ is $Mm=Mp$.  By the above claim, we deduce that $[G,G]=[G_M,G_M] \times \SL_2(\ZZ_p) \times \prod_{\ell\nmid Mp} \SL_2(\ZZ_\ell)$ which contradicts that $p$ divides the level of $[G,G]$.

Therefore, $m=15$ and $p=5$.  We can view $[G_{15M},G_{15M}]$ as a subgroup of $[G_{5M},G_{5M}] \times [G_3,G_3]$ whose projection to each factor is surjective.  By Goursat's lemma (Lemma~\ref{L:Goursat}), there are normal subgroups $B_1'$ and $B_2'$ of $[G_{5M},G_{5M}]$ and $[G_3,G_3]$, respectively, so that the image of $[G_{15M},G_{15m}]$ in $[G_{5M},G_{5M}]/B_1' \times [G_{3},G_{3}]/B_2'$ is the graph of an isomorphism.  The group $G_3=\GL_2(\ZZ_3)$ is prosolvable (since $\GL_2(\ZZ/3\ZZ)$ is solvable) and $\SL_2(\ZZ_5)$ is equal to its own commutator subgroup by Lemma~\ref{L:comm}(\ref{L:comm ii}), so we must have $\{I\}\times \SL_2(\ZZ_5) \subseteq B_1'$.   From this we deduce that the level of $[G_{15M},G_{15M}]\subseteq \SL_2(\ZZ_{15M})$ is not divisible by $5$ which contradicts that $m=15$.  We conclude that the case $B_2\supseteq \SL_2(\ZZ_p)$ does not occur.

\noindent $\bullet$ 
We have now proved that parts (\ref{L:maximal leftovers i}) and (\ref{L:maximal leftovers iii}) hold with our prime $p$.  When $p=3$, we have $G_p=\GL_2(\ZZ_p)$.   Now assume that $p=5$ and $G_5=\GL_2(\ZZ_5)$.  From (\ref{L:maximal leftovers iii}), we find that $\mathfrak{S}_5=\psi_5(G_5)$ is isomorphic to a quotient of $\calG_M$.

We claim that there is an open normal subgroup $W$ of $\calG_M$ for which $\calG_M/W$ is isomorphic to a direct product of copies of the simple group $\mathfrak{A}_5$ and the finite quotients of $W$ have no simple groups isomorphic to $\mathfrak{A}_5$ occurring as composition factors.  This follows from the description of maximal subgroups of $\GL_2(\FF_\ell)$ for $\ell|M$, see \cite[\S2]{Serre-Inv72}, which shows that if a subgroup of $\GL_2(\FF_\ell)$ has $\mathfrak{A}_5$ as a composition factor, then its image in $\operatorname{PGL}_2(\FF_\ell)$ is isomorphic to $\mathfrak{A}_5$ (note that $\mathfrak{A}_5\not\cong \SL_2(\FF_\ell)/\{\pm \ell\}$ since $\ell\neq 5$).

So $W$ must lie in the kernel of our surjective homomorphism $\calG_M\to \mathfrak{S}_5$ which is impossible since $\calG_M/W$ is isomorphic to a direct product of copies of $\mathfrak{A}_5$.   Therefore, $G_5\neq \GL_2(\ZZ_5)$.  Since $G_5\supseteq \ZZ_p^\times \SL_2(\ZZ_p)$, part (\ref{L:maximal leftovers ii}) follows.
\end{proof}

\subsection{Special subgroups} \label{SS:special subgroups}

Let $\calG$ be an open subgroup of $\GL_2(\Zhat)$ that contains the scalars $\Zhat^\times  I$.   Fix an open group $W$ of $\SL_2(\Zhat)$ that satisfies 
\[
[\calG,\calG] \subseteq W \subseteq \calG \cap\SL_2(\Zhat).
\]
The group $W$ is normal in $\calG$ and $\calG/W$ is abelian since $[\calG,\calG]\subseteq W \subseteq \calG$.  

Given an open subgroup $U$ of $\det(\calG)$, the following theorem gives a criterion that determines whether there exists an open subgroup $G$ of $\calG$ for which $G\cap \SL_2(\Zhat)=W$ and $\det(G)=U$.   Let $N$ be the least common multiple of the levels of $\calG$ and $W$ in $\GL_2(\Zhat)$ and $\SL_2(\Zhat)$, respectively.   Define $N_1:=N$ if $N$ is odd and $N_1:=\operatorname{lcm}(N,8)$ if $N$ is even. 

\begin{thm} \label{T:G existence}
Let $U$ be an open subgroup of $\det(\calG) \subseteq \Zhat^\times$ and let $S:=U_N[2^\infty]$ be the $2$-power torsion subgroup of $U_N\subseteq \ZZ_N^\times$.    Then the following are equivalent:
\begin{alphenum}
\item \label{T:G existence a}
There is an open subgroup $G\subseteq \calG$ with $G\cap \SL_2(\Zhat)=W$ and $\det(G)=U$.
\item \label{T:G existence b}
There is a homomorphism $\beta\colon S \to \calG_N/W_N$ such that $\det(\beta(a))=a$ for all $a\in S$.
\item \label{T:G existence c}
There is a homomorphism $\beta\colon S\to \calG(N_1)/W(N_1)$ such that $\det(\beta(a))\equiv a \pmod{N_1}$ for all $a\in S$.
\end{alphenum}
Moreover, if a group $G$ as in (\ref{T:G existence a}) exists, then there is such a group whose level divides a power of $2$ times the least common multiple of $N$ and the level of $U \subseteq \Zhat^\times$. 
\end{thm}
\begin{proof}
Define $U'=U_N\times \prod_{\ell\nmid N} \ZZ_\ell^\times$.  We have $U\subseteq U' \subseteq \det(\calG)$.  Note that the conditions (\ref{T:G existence b}) and (\ref{T:G existence c}) depend only on $U_N=U'_N$.   If there is an open subgroup $G'\subseteq \calG$ satisfying $G'\cap \SL_2(\Zhat)=W$ and $\det(G')=U'$, then the group $G:=\{g\in G' : \det(g)\in U\}$ will satisfy (\ref{T:G existence a}).   Also if the level of $G'$ divides an integer $m$, then the level of $G$ will divide the least common multiple of $m$ and the level of $U$.    So without loss of generality, we may assume that $U=U_N\times \prod_{\ell\nmid N} \ZZ_\ell^\times$.

We first assume there is a homomorphism $\beta\colon S \to \calG_N/W_N$ as in (\ref{T:G existence b}).    Recall that for odd $\ell$, $\ZZ_\ell^\times = C (1+\ell\ZZ_\ell)$ for a finite cyclic group $C$ of order $\ell-1$ and $1+\ell\ZZ_\ell \cong \ZZ_\ell$.  We have $\ZZ_2^\times =\pm (1+8\ZZ_2)$ and $1+8\ZZ_2\cong \ZZ_2$.   Since $U_N$ is an open subgroup of $\ZZ_N^\times=\prod_{\ell|N}\ZZ_\ell^\times$, we have an internal direct product of groups 
\begin{align} \label{E:internal S A A}
U_N = S \cdot A_1 \cdot A_2,
\end{align}
where $A_1$ is torsion-free $\ZZ_2$-module of rank at most $1$ and $A_2$ is isomorphic to a  product of an odd finite abelian group with $\prod_{\ell|N,\ell \neq 2} \ZZ_\ell$.   Fix a $u_1\in A_1$ that generates $A_1$ as a $\ZZ_2$-module and choose an element $g_1\in \calG_N$ for which $\det(g_1)=u_1$.  There is a unique continuous homomorphism $t_1\colon A_1 \to \calG_N/W_N$ such that $t_1(u_1)=g_1 W_N$.  Since $\det(g_1)=u_1$, we have $\det(t_1(a))=a$ for all $a\in A_1$.   The map $A_2\to A_2$, $a\mapsto a^2$ is an isomorphism of groups whose inverse we denote by $\psi$.   Define the homomorphism $t_2\colon A_2 \to \calG_N/W_N$, $a\mapsto (\psi(a)\cdot I) \cdot W_N$; it satisfies $\det(t_2(a))=\psi(a)^2=a$ for all $a\in A_2$.  Using the direct product (\ref{E:internal S A A}) with the maps $\beta\colon S\to \calG_N/W_N$, $t_1$ and $t_2$, we obtain a homomorphism 
$s_N\colon U_N \to \calG_N/W_N$ that satisfies $\det(s_N(a))=a$ for all $a\in U_N$.   For each prime $\ell\nmid N$, we define the homomorphism $s_\ell\colon \ZZ_\ell^\times \to \GL_2(\ZZ_\ell)/\SL_2(\ZZ_\ell)$ by $a\mapsto \left(\begin{smallmatrix} 1 & 0 \\0 & a\end{smallmatrix}\right)\cdot \SL_2(\ZZ_\ell)$.   The map $s_\ell$ is an isomorphism with inverse given by the determinant; in particular, $\det(s_\ell(a))=a$ for all $a\in \ZZ_\ell^\times$.   By combining $s_N$ with the $s_\ell$ for $\ell\nmid N$, we obtain a homomorphism
\[
s\colon U= U_N\times {\prod}_{\ell\nmid N} \ZZ_\ell^\times \to   \calG_N/W_N \times {\prod}_{\ell\nmid N} \GL_2(\ZZ_\ell)/\SL_2(\ZZ_\ell) = \calG/W
\]
that satisfies $\det(s(a))=a$ for all $a\in U$ (this uses that the levels of $\calG$ and $W$ are not divisible by any prime $\ell\nmid N$).  There is a unique subgroup $G$ of $\GL_2(\Zhat)$ with $G\supseteq W$ for which $G/W$ is equal to $s(U)\subseteq \calG/W$.  The group $G$ is closed in $\GL_2(\Zhat)$ since $s$ is continuous.  We have $\det(G)=\det(s(U))=U$.  Since $\det(s(a))=a$ for all $a\in U$, we deduce that $G\cap \SL_2(\Zhat)=W$.  Using that $\det(G)=U$ is open in $\Zhat^\times$ and $G\cap \SL_2(\Zhat)=W$ is open in $\SL_2(\Zhat)$, we find that $G$ is an open subgroup of $\GL_2(\Zhat)$.

We claim that the level of $G$ divides a power of $2$ times $N$.  From the definition of $s$ and our $s_\ell$, we find that $G\supseteq \{I\} \times \prod_{\ell\nmid N} \GL_2(\ZZ_\ell)$.  Therefore, the level of $G$ is not divisible by any prime $\ell\nmid N$.   We have $G_N \supseteq W_N \supseteq \{B\in \SL_2(\ZZ_N) : B \equiv I \pmod{N}\}$.  Using our choice of $t_2$, we find that $G_N$ contains the scalar matrices $c I$ for all $c$ in the set $\{a \in U_N : a \equiv 1 \pmod{N}\} \cap (\{1\} \times \prod_{\ell|N,\ell\neq 2} \ZZ_\ell^\times)$.  Therefore,
\[
G_N\supseteq {\prod}_{\ell|N, \ell=2} \{I\} \times {\prod}_{\ell|N,\ell\neq 2} H_\ell,
\]
where $H_\ell:=(1+\ell^{e_\ell} \ZZ_\ell) \{B\in \SL_2(\ZZ_\ell) :  B \equiv I \pmod{\ell^{e_\ell}} \}$ and $\ell^{e_\ell}$ is the largest power of $\ell$ dividing $N$.   Take any odd prime $\ell|N$.  To complete the proof of the claim, it suffices to show that $H_\ell \supseteq \{B\in \GL_2(\ZZ_\ell) : B\equiv I \pmod{\ell^{e_\ell}}\}$.   Take any $B\in \GL_2(\ZZ_\ell)$ with  $B\equiv I \pmod{\ell^{e_\ell}}$.   We have $\det(B) \in 1 +\ell^{e_\ell} \ZZ_\ell = (1 +\ell^{e_\ell} \ZZ_\ell)^2$, where the equality uses that $\ell$ is odd.   So there is a $u \in 1+\ell^{e_\ell}\ZZ_\ell$ for which $\det(B)=u^2$.   Define $C:=u^{-1} B \in \SL_2(\ZZ_\ell)$ and note that $C\equiv I \pmod{\ell^{e_\ell}}$.   So $B=u C$ is in $H_\ell$ and the claim follows.

This completes the proof that (\ref{T:G existence b}) implies (\ref{T:G existence a}).   After we prove the reverse implication, the final statement of the theorem will follow from the above claim.

Now suppose that there is a group $G\subseteq \calG$ satisfying the properties of (\ref{T:G existence a}).   Since $G\cap\SL_2(\Zhat)=W$ and $\det(G)=U$, the map $\det\colon G/W \to U$ is an isomorphism of groups whose inverse gives rise to a homomorphism  $s\colon U\to G/W \subseteq \calG/W$ that satisfies $\det(s(a))=a$ for all $a\in U$.   Let $\iota\colon U_N \to U_N\times \prod_{\ell\nmid N} \ZZ_\ell^\times =U$ be the homomorphism that is the identity on the $U_N$ factor and trivial on the $\ZZ_\ell^\times$ factors.   Define the homomorphism $s_N\colon U_N \hookrightarrow U \xrightarrow{s} \calG/W \to \calG_N/W_N$, where the first map is $\iota$ and the last map is the $N$-adic projection.  We have $\det(s_N(a))=a$ for all $a\in U_N$.  We have $S\subseteq U_N$, so $\beta:=s_N|_{S}\colon S \to \calG_N/W_N$ is a homomorphism satisfying $\det(\beta(a))=a$ for all $a\in S$.  This completes the proof that (\ref{T:G existence a}) implies (\ref{T:G existence b}).

Now suppose that (\ref{T:G existence c}) holds with a homomorphism $\beta\colon S\to \calG(N_1)/W(N_1)$.

For any $a\in S$, we claim that there is a $g\in \calG_N$ such that $\det(g)=a$ and such that the order of $a$ agrees with the order of $g W_N$ in $\calG_N/W_N$.   Take any $a\in S$ and denote its order by $e$.  We may assume that $e\geq 2$ since we can take $g=I$ when $e=1$. Choose a $g_1 \in \calG_N$ whose image modulo $N_1$ represents the coset $\beta(a)\in \calG(N_1)/W(N_1)$.   We have $\det(g_1)\equiv \det(\beta(a))\equiv a \pmod{N_1}$.   Since $N_1\equiv 0 \pmod{8}$ when $N_1$ is even, we have $(1+N_2\ZZ_N)^2=1+N_1 \ZZ_N$ where $N_2:=N_1=N$ if $N$ is odd and $N_2:=N_1/2$ if $N$ is even.  Therefore, $\det(g_1) a^{-1} =c^{-2}$ for some $c\in 1+N_2\ZZ_N$.  Define $g:=c g_1$; it lies in $\calG_N$ since $\calG$ contains the scalars in $\GL_2(\Zhat)$.  We have $\det(g)=c^2 \det(g_1)= a$.    The matrix $g^e$ thus lies in $\calG_N\cap \SL_2(\ZZ_N)$.   We have $\beta(a)^e= \beta(a^e)=1$, so reducing $g_1^e$ modulo $N_1$ gives the identity coset in $\calG(N_1)/W(N_1)$.  We have $c^e \equiv 1 \pmod{N_1}$ since $c\equiv 1 \pmod{N_2}$ and $e>1$ is a power of $2$.  Therefore, $g^e=c^e g_1^e$ modulo $N_1$ lies in $W(N_1)$.   Since $W$ has level dividing $N$ and $g^e\in \SL_2(\ZZ_N)$, we deduce that $g^e\in W_N$.  So $g W_N$ has order at most $e$ in $\calG_N/W_N$.  The order is exactly $e$ since $\det(g)=a$ has order $e$ in $U_N$.  This completes the proof of the claim.

Let $\{u_1,\ldots, u_r\}$ be a minimal generating set of the finite abelian group $S$ and let $e_i$ be the order of $u_i$.   In particular, we have an isomorphism $\ZZ/e_1\ZZ \times \cdots \times \ZZ/e_r \ZZ \to S$, $(n_1,\ldots, n_r)\mapsto u_1^{n_1}\cdots u_r^{n_r}$.  So to define a homomorphism $S\to \calG_N/W_N$ as in (\ref{T:G existence b}) we need only map each $u_i$ to an element in $\calG_N/W_N$ of order $e_i$ that has determinant $u_i$.   Therefore, (\ref{T:G existence b}) is true by the previous claim.   This proves that (\ref{T:G existence c}) implies (\ref{T:G existence b}).

Finally, it remains to show that (\ref{T:G existence b}) implies (\ref{T:G existence c}).   This is clear by taking any homomorphism as in (\ref{T:G existence b}) and composing with the quotient map $\calG_N/W_N\to \calG(N_1)/W(N_1)$.
\end{proof}

\begin{remark} \label{R:G existence}
We note that the condition (\ref{T:G existence c}) in Theorem~\ref{T:G existence} is straightforward to check in practice since all the groups involved are finite.  Now suppose that the conditions of Theorem~\ref{T:G existence} hold.   One way to find a group $G$ as in (\ref{T:G existence a}), if it exists, is to do a direct search modulo $2^i N$ for $i=0,1,\ldots$ (the proof of Theorem~\ref{T:G existence} also gives a constructive way when starting with homomorphism $\beta$ as in (\ref{T:G existence c})).
\end{remark}

\subsection{Proof of Theorem~\ref{T:nice supply of G}} \label{SS: nice supply proof}

The group $[\calG,\calG]$ is open in $\calG\cap \SL_2(\Zhat)$ by Lemma~\ref{L:openness of commutator}.  Let $N$ be the least common multiple of the levels of $\calG$ and $[\calG,\calG]$.  Note that $N$ is even since the level of $[\calG,\calG]$ is even.  Let $W$ be any of the finitely many of groups satisfying $[\calG,\calG]\subseteq W \subseteq \calG \cap \SL_2(\Zhat)$; its level divides $N$.   To prove the theorem, we need to show that one can find an open subgroup $G\subseteq \calG$ such that $G\cap \SL_2(\Zhat)=W$, $[\Zhat^\times: \det(G)]$ is a power of $2$, and the level of $G$ divides $N$ times a power of $2$.

There is an open subgroup $U_0 \subseteq \det(\calG_N) \subseteq \ZZ_N^\times$ with $U_0[2^\infty]=1$ such that $\det(\calG_N)$ is generated by its $2$-power torsion and $U_0$; we can may further choose $U_0$ so that it contains the open and torsion-free subgroup $\{a\in \ZZ_N^\times : a \equiv I \pmod{2N}\}$.   Let $S$ be a subgroup of the $2$-power torsion of $\det(\calG_N) \subseteq \ZZ_N^\times$, with maximal cardinality, for which condition (\ref{T:G existence b}) of Theorem~\ref{T:G existence} holds, cf.~Remark \ref{R:G existence}.    Define $U:=(S\cdot U_0) \times {\prod}_{\ell\nmid N} \ZZ_\ell^\times$;
it is an open subgroup of $\det(\calG)$ and $[\det(\calG):U]$ is a power of $2$.  The level of $U_0$, and hence also of $U$, divides $2N$.   By our choice of $U$, Theorem~\ref{T:G existence} implies that there is an open subgroup $G \subseteq \calG$ such that $G\cap\SL_2(\Zhat)=W$, $\det(G)=U$, and the level of $G$ divides a power of $2$ times $N$.   By computing in $\GL_2(\ZZ/2^i N \ZZ)$ for $i\geq 0$, we can find such a group $G$.

\section{Proof of the theorems from \S\ref{SS:agreeable closure overview}}  \label{S:proofs of agreeable closure overview}

Note that our proofs of the finiteness of $\calA_1$ and $\calA_2$ will be given in a manner so that it is clear that they are indeed computable.

\subsection{Proof of Theorem~\ref{T:main agreeable 1}} \label{SS:proof main agreeable}

First consider any agreeable subgroup $G$ of $\GL_2(\Zhat)$ for which $X_G$ has genus at most $1$.   Define $H:= G \cap \SL_2(\Zhat)$; it is an open subgroup of $\SL_2(\Zhat)$.   The group $H$ contains $-I$ since $G$ contains all the scalars in $\GL_2(\Zhat)$.  Let $N$ be the level of $H$.   Define the congruence subgroup $\Gamma_G:=\SL_2(\ZZ) \cap H=\SL_2(\ZZ)\cap G$ of $\SL_2(\ZZ)$; equivalently, it is the congruence subgroup of level $N$ whose image modulo $N$ agrees with the image of $H$ modulo $N$.   In particular, $H$ can be recovered from $\Gamma_G$ and we have $-I \in \Gamma_G$.   The genus of $\Gamma_G$ agrees with the genus of $G$, see Remark~\ref{R:modular curve follow-up}, and hence is at most $1$.

There are only finitely many congruence subgroups of $\SL_2(\ZZ)$ that have genus at most $1$ and contain $-I$, cf.~\cite{MR2016709}.   Moreover, all such congruence subgroups are explicitly given in \cite{MR2016709} up to conjugacy in $\GL_2(\ZZ)$; there are $121$ and $163$ conjugacy classes with genus $0$ and $1$, respectively.  

Now fix one of the finitely many congruence subgroups $\Gamma$ of $\SL_2(\ZZ)$ that have genus at most $1$ and contains $-I$.    Let $H$ be the open subgroup of $\SL_2(\Zhat)$ corresponding to $\Gamma$.    We have $-I \in H$.  In the rest of the proof, we will explain how to compute the (finitely many) agreeable subgroups $G$ of $\GL_2(\Zhat)$ for which $G\cap\SL_2(\Zhat)=H$.   The finiteness of $\calA_1'$, and hence also of $\calA_1$, will be obtained by varying over the finite many $\Gamma$.    Let $N$ be the level of $\Gamma$; it is also the level of $H$.   Define the integer $N_1:= 2 \operatorname{lcm}(N,12)$.   

\begin{lemma} \label{L:searching for agreeable}
For any agreeable subgroup $G$ of $\GL_2(\Zhat)$ with $G\cap \SL_2(\Zhat)=H$, the level of $G$ divides $N_1$.
\end{lemma}
\begin{proof} 
Define $N_0=\operatorname{lcm}(N,12)$.  We have $H=H_{N_0} \times \prod_{\ell\nmid N_0} \SL_2(\ZZ_\ell)$ since the level of $H$ divides $N_0$.  Hence $[H,H]=[H_{N_0},H_{N_0}] \times \prod_{\ell\nmid N_0} \SL_2(\ZZ_\ell)$ by Lemma~\ref{L:comm}(\ref{L:comm ii}).    In particular, the level of $[H,H]$ is divisible only by primes dividing $N_0$; equivalently, dividing $N_1=2N_0$.
Consider any agreeable subgroup $G$ of $\GL_2(\Zhat)$ for which $G\cap \SL_2(\Zhat)=H$.   We have $\SL_2(\Zhat)\supseteq [G,G]\supseteq [H,H]$, so any prime dividing the level of $[G,G]$ must also divide $N_1$.    Since $G$ is agreeable, we have $G=G_{N_1} \times \prod_{\ell \nmid N_1} \GL_2(\ZZ_\ell)$.  It remains to show that the level of $G_{N_1} \subseteq \GL_2(\ZZ_{N_1})$ divides $N_1$.   From Lemma~\ref{L:level by adjoining scalars} and our choice of $N_1$, we find that $\ZZ_{N_1}^\times  H_{N_1}$ is an open subgroup of $\GL_2(\ZZ_{N_1})$ whose level divides $N_1$.  We have $\ZZ_{N_1}^\times  H_{N_1} \subseteq G_{N_1} $ since $G$ contains $H$ and $\Zhat^\times\cdot I$, and hence the level of $G_{N_1}$ divides $N_1$.
\end{proof}

We now describe how to compute all the agreeable subgroups $G$ of $\GL_2(\Zhat)$ for which $G\cap \SL_2(\Zhat)=H$.   By Lemma~\ref{L:searching for agreeable}, the level of such a group $G$ divides $N_1$.   So we first look for subgroups $\bbar{G}$ of $\GL_2(\ZZ/N_1\ZZ)$ for which $\bbar{G} \cap \SL_2(\ZZ/N_1\ZZ)$ equals the the image of $H$ modulo $N_1$.    There are only finitely many such groups $\bbar{G}$ which give rise to finite many candidate groups $G$ of $\GL_2(\Zhat)$ which satisfy $G\cap \SL_2(\Zhat)=H$.   We can then check which of the candidates $G$ are agreeable.

Finally, suppose there is a group $G \in \calA_1$ and a prime $\ell>19$ for which $\ell$ divides the level of $[G,G]$.  By Lemma~\ref{L:old agreeable}, we have $G_\ell \not\supseteq \SL_2(\ZZ_\ell)$ and hence the level of $G\cap \SL_2(\Zhat)\subseteq \SL_2(\Zhat)$ is divisible by $\ell$.  From our argument above, we find that there is a congruence subgroup $\Gamma \subseteq \SL_2(\ZZ)$ of genus at most $1$ for which $\ell$ divides the level of $\Gamma$.  However, the classification of low genus congruence subgroups in \cite{MR2016709} shows that $19$ is the largest possible prime divisor of the level of a congruence subgroup of genus at most $1$.   Therefore, the level of $[G,G]$ is not divisible by any prime $\ell >19$ for all $G\in \calA_1$.

\subsection{Proof of Theorem~\ref{T:main agreeable 2}} \label{SS:proof main agreeable 2}

First consider any agreeable subgroup $G$ of $\GL_2(\Zhat)$ with genus at least $2$ that satisfies $G_\ell \supseteq \SL_2(\ZZ_\ell)$ for all $\ell>19$.   By Lemma~\ref{L:old agreeable}, the level of $[G,G]$ is not divisible by any prime $\ell>19$.   Since $G$ is agreeable, its level is not divisible by any prime $\ell>19$.   Choose a maximal agreeable group $G\subseteq G' \subseteq \GL_2(\Zhat)$ for which $G'$ has genus at least $2$.   Since the level of $G'$ divides the level of $G$, we deduce that $G'$ lies in $\calA_2'$.  This proves part (\ref{T:main agreeable 2 ii}).

Now take any group $G\in \calA_2'$.  Choose a minimal agreeable group $\calG$ of genus at most $1$ that satisfies $G\subseteq \calG$.  Using the definition of $\calA_1'$ and $\calA_2'$, we find that $\calG\in \calA_1'$ and that $G$ is a (proper) maximal agreeable subgroup of $\calG$.   Since we are only interested in groups up to conjugacy, we may assume that $\calG \in \calA_1$.  Since $\calA_1$ is finite by Theorem~\ref{T:main agreeable 1}, to prove the finiteness of $\calA_2$ it suffices to show that every group $\calG\in \calA_1$ has only finitely many maximal agreeable subgroups whose level is not divisible by any prime $\ell>19$.

Fix a group $\calG \in \calA_1$ and let $M$ be the product of the primes that divide the level of $[\calG,\calG]$.  We have $\ell\nmid M$ for all $\ell>19$ by Theorem~\ref{T:main agreeable 1}.   Since $\calG$ is agreeable, the level of $\calG$ is also not divisible by any prime $\ell>19$.   We now consider the maximal agreeable subgroups $G$ of $\calG$ as classified in \S\ref{SS:maximal agreeable}.   We want to show that there are only finitely many of each type and make clear that they are computable.
  
  If $3\nmid M$, we obtain a single maximal agreeable subgroup as in Lemma~\ref{L:obvious maximal}(\ref{L:obvious maximal iii}).

Take any prime $p\nmid M$ with $p \leq 19$.   The maximal open subgroups $B\subseteq \GL_2(\ZZ_p)$ with $\ZZ_p^\times I \subseteq B$ and $\SL_2(\ZZ_p) \not\subseteq B$, give rise to the maximal agreeable subgroups of $\calG$ as in Lemma~\ref{L:obvious maximal}(\ref{L:obvious maximal ii}).   By Lemma~\ref{L:Serre mod ell to ell-adic}, the group $B$ have level $p$ and are thus easy to enumerate.

A maximal agreeable subgroup of $\calG$ as given in Lemma~\ref{L:obvious maximal}(\ref{L:obvious maximal i}) arises from a maximal proper open subgroup $B$ of $\calG_M$ that contain $\ZZ_M^\times I$.   Let $N$ be the least common multiple of $4$, $M$, and the level of $\calG$.  Lemma~\ref{L:how to look for maximal} implies that the level of $B$ divides $N \ell$ for some prime $\ell | N$.  So one need only look for maximal subgroups of the image of $G$ in $\GL_2(\ZZ/N\ell\ZZ)$ for each $\ell|N$.

Now consider any prime $p\in\{3,5\}$ that does not divide $M$.   We now consider maximal agreeable subgroups of $\calG$ as described in Lemma~\ref{L:maximal leftovers}.   By Lemma~\ref{L:maximal leftovers}, it suffices to compute the open normal subgroups of $\calG_M$ for which the quotient is isomorphic to a group $Q \in \{\mathfrak{S}_p,\mathfrak{A}_p\}$ where $Q\neq \mathfrak{A}_p$ when $p=3$.     Let $N$ be the least common multiple of $M$ and the level of $\calG$; it has the same prime divisors as $M$.   For any continuous and surjective homomorphism $\calG_M \twoheadrightarrow Q$, the kernel contains all $g\in \calG_M$ with $g\equiv I \pmod{N}$ since $Q$ contains no normal $\ell$-groups for all $\ell\nmid M$.  Therefore, one need only look for normal subgroup of the image of $\calG_M$ modulo $N$ that have $Q$ as a quotient group.

\subsection{Proof of Theorem~\ref{T:loose agreeable}}

Let $\calG$ be the agreeable closure of $G_E$.  Proposition~\ref{P:key property} implies that $G$ is a minimal element of $\calA_1'$, with respect to inclusion, for which $G_E$ is conjugate in $\GL_2(\Zhat)$ to a subgroup of $G$.   By conjugating $G_E$, we may assume that $G_E \subseteq G$.   Since $\calG$ is the minimal agreeable subgroup containing $G_E$, we have $G_E \subseteq \calG \subseteq G$. If $\calG$ has genus at most $1$, then $\calG=G$ since otherwise $G$ is not a minimal element of $\calA_1'$ with respect to inclusion.  We can now assume that $\calG$ has genus at least $2$.

We claim that $\calG_\ell\supseteq \SL_2(\ZZ_\ell)$ for all primes $\ell>19$.   Take any prime $\ell>19$.  By assumption, we have $\rho_{E,\ell}(\Gal_K)\supseteq \SL_2(\ZZ/\ell\ZZ)$ and hence $(G_E)_\ell \supseteq \SL_2(\ZZ_\ell)$ by Lemma~\ref{L:Serre mod ell to ell-adic}.   The claim follows since $G_E\subseteq \calG$ and hence $(G_E)_\ell \subseteq \calG_\ell$.

Theorem~\ref{T:main agreeable 2}(\ref{T:main agreeable 2 ii}) implies that $\calG$, and hence also $G_E$, is conjugate in $\GL_2(\Zhat)$ to a subgroup of some $G'\in \calA_2$.  Proposition~\ref{P:key property} implies that $K_{G'} \subseteq K$ and $j_E \in \pi_{G'}(X_{G'}(K))\subseteq J_K$, where the last inclusion uses that $G'$ lies in $\calA_2$.   Since $j_E\notin J_K$ by assumption, the case where $\calG$ has genus at least $2$ does not occur.

\end{document}